\documentclass[11pt]{article}
\usepackage{amsfonts,latexsym,amssymb,amsthm,amsmath,graphicx,cases,mathrsfs}
\usepackage[ulem=normalem,draft]{changes}
\usepackage{paralist}
\usepackage{graphicx} 
\usepackage{epsfig}
\usepackage{epstopdf}
\usepackage{ulem}
\usepackage{cases}
\usepackage[titletoc,title]{appendix}
\usepackage{empheq}
\usepackage{tikz}
\usetikzlibrary{arrows,shapes}
\usetikzlibrary{decorations.pathmorphing,decorations.pathreplacing}
\usetikzlibrary{calc,patterns,angles,quotes}
\usepackage{todonotes}
\usepackage{multicol}
\usepackage{cancel}
\usepackage{float}

\usepackage[colorlinks=true]{hyperref}
\hypersetup{urlcolor=blue, citecolor=red}
\allowdisplaybreaks
\graphicspath{ {./images/} }

\newtheorem{theorem}{Theorem}[section]

\newtheorem{thm}{Theorem}

\newtheorem{lemma}[theorem]{Lemma}

\theoremstyle{definition}

\newtheorem{rmk}{Remark}

\newcommand{\be}{\begin{equation}}
\newcommand{\ee}{\end{equation}}
\newcommand{\bsubeq}{\begin{subequations}}
	\newcommand{\esubeq}{\end{subequations}}
\renewcommand{\div}{\text{div}}
\newcommand{\ds}{\displaystyle}

\newcommand{\calL}{{\mathcal{L}}}
\newcommand{\calH}{{\mathcal{H}}}

\newcommand{\calD}{{\mathcal{D}}}

\newcommand{\calM}{{\mathcal{M}}}

\newcommand{\calS}{{\mathcal{S}}}

\newcommand{\calO}{{\mathcal{O}}}

\newcommand{\BA}{\mathbb{A}}

\newcommand{\BR}{\mathbb{R}}

\newcommand{\BW}{\mathbb{W}}

\newcommand{\wti}{\widetilde}

\newcommand{\bpm}{\begin{pmatrix}}
	\newcommand{\epm}{\end{pmatrix}}

\newcommand{\bbm}{\begin{bmatrix}}
	\newcommand{\ebm}{\end{bmatrix}}

\newcommand{\bem}{\begin{matrix}}
	\newcommand{\eem}{\end{matrix}}

\numberwithin{equation}{section}
\numberwithin{thm}{section}
\numberwithin{rmk}{section}
\numberwithin{prop}{section}

\newcommand{\bs}[1]{\mathbf{#1}}
\newcommand{\bos}[1]{\boldsymbol{#1}}

\newcommand{\Lso}{\bs{L}^{q}_{\sigma}(\Omega)}

\newcommand{\Lo}[1]{\bs{L}^{#1}_{\sigma}(\Omega)}

\newcommand{\Ls}{\bs{L}^{q}_{\sigma}}

\newcommand{\Lqpso}{\bs{L}^{q'}_{\sigma}(\omega)}

\newcommand\rfrac[2]{{}^{#1}\!/_{#2}}

\newcommand{\norm}[1]{\left\lVert#1\right\rVert}
\newcommand{\abs}[1]{\left\lvert#1\right\rvert}

\newcommand{\nin}{\noindent}

\newcommand{\curl}{curl}

\newcommand{\LpO}{\bs{L}^p(\Omega)}

\newcommand{\Yqs}{\bs{Y}^q_{\sigma}(\Omega)}
\newcommand{\Yps}{\bs{Y}^{q'}_{\sigma}(\Omega)}

\newcommand{\Yqss}{(\bs{Y}^q_{\sigma}(\Omega))^*}

\renewcommand{\curl}{\operatorname{curl}}
\renewcommand{\div}{\operatorname{div}}

\newcommand{\thistheoremname}{}

\newtheorem*{genericthm*}{\thistheoremname}
\newenvironment{namedthm*}[1]
{\renewcommand{\thistheoremname}{#1}%
	\begin{genericthm*}}
	{\end{genericthm*}}

\usepackage{todonotes}

\providecommand{\keywords}[1]
{
	\small	
	\textbf{\textit{Keywords---}} #1
}

\begin{document}

\title{Unique Continuation of Static Over-Determined Magnetohydrodynamic Equations}

\author{Irena Lasiecka \thanks{Department of Mathematical Sciences, University of Memphis, Memphis, TN 38152 USA and IBS, Polish Academy of Sciences, Warsaw, Poland. (lasiecka@memphis.edu)}
	\and Buddhika Priyasad \thanks{Department of Mathematics and Statistics, University of Konstanz, Konstanz, Germany. (priyasad@uni-konstanz.de).}
	\and Roberto Triggiani \thanks{Department of Mathematical Sciences, University of Memphis, Memphis, TN 38152 USA. (rtrggani@memphis.edu)}}

\date{}

\maketitle

\begin{abstract}
	This paper establishes the Unique Continuation Property (UCP) for a suitably overdetermined Magnetohydrodynamics (MHD) eigenvalue problem, which is equivalent to the Kalman, finite rank, controllability condition for the finite dimensional unstable projection of the linearized dynamic MHD problem. It is the ``ignition key" to obtain uniform stabilization of the dynamic nonlinear MHD system near an unstable equilibrium solution, by means of finitely many, interior, localized feedback controllers \cite{LPT.6}. The proof of the UCP result uses a pointwise Carleman-type estimate for the Laplacian following the approach that was introduced in \cite{RT2:2009} for the Navier-Stokes equations and further extended \cite{TW:2021} for the Boussinesq system.
\end{abstract}

\nin \keywords{Magnetohydrodynamics equations, Unique continuation, Uniform stabiliztion, Carleman estimates.}

\section{Introduction, the role of the unique continuation in uniform stabilization, main results, literature.}

\subsection{Unique continuation properties (UCP) of over-determined static problems: the ignition key for uniform feedback stabilization}

\noindent \textbf{The dynamic MHD equations.} After the initial work carried out by the Nobel laureate Hannes Alfv\'{e}n in 1970, Magnetohydrodynamics (henceforth referred to as MHD) has culminated as an emerging discipline in Plasma physics. MHD refers to phenomena arising in electrically conducting magnetic fluids. It is caused by the induction of current in a conductive fluid flow due to a magnetic field and moreover by polarization of the fluid and reciprocal changes in the magnetic field. MHD has been used extensively in plasma confinement, liquid metal cooling of nuclear reactors and electromagnetic casting (EMC). The system of MHD equations - below in \eqref{1.1} - consists of the Navier-Stokes equations of a viscous incompressible fluid flow suitably coupled by high-order coupling with Maxwell-Ohm equations (of parabolic character) of an electromagnetic field \cite{Sol:1960, ST:1983, YG:1998, Yam:2004}.\\

\noindent \textbf{Uniform stabilization of fluids.} Recent work by the authors has focused on the problem of feedback stabilization (asymptotic turbulent suppression) of fluids such as Navier-Stokes equations as well as Boussinesq systems, and MHD equations, on bounded domains $\Omega$ in $\BR^d, \, d = 2,3$ by means of finite dimensional, static, feedback controls, either localized in the interior \cite{LPT.1, LPT.3, LPT.6} or else localized at the boundary \cite{LPT.2, LPT.4}. More specifically, to illustrate, a 20-year old open problem (introduced by A. Fursikov around 2000 \cite{F.1, F.2}) as to whether the 3d-Navier-Stokes equation could be stabilized by static, boundary based localized feedback controllers which moreover are finite dimensional (in fact, of minimal dimension), was positively proved in \cite{LPT.2}. It required abandoning the usual Sobolev-Hilbert setting of the literature in favor of a new Besov space setting with tight indices (``close" to $L^3(\Omega)$ for $d = 3$).\\

\noindent \textbf{Critical role of UCP.} Following the strategy for feedback stabilization of parabolic dynamics introduced in \cite{RT:1975}, an extensively used since in the literature, a first step of the analysis consists in feedback stabilizing with an arbitrarily large decay rate, the finite dimensional unstable component of the full dynamics with static, state feedback controls. The ability to do so requires the property of controllability of the finite dimensional unstable component, which in fact is equivalent to the needed ``spectrum allocation property" \cite{Za:1992}. Showing such controllability property (Kalman's algebraic rank condition) for PDE problems requires a fundamental Unique Continuation Property (UCP). Numerous illustrations from classical parabolic problems to fluid (Navier-Stokes, Boussinesq system) are given in the extensive article \cite{RT:2024}, which employs UCP for fluids \cite{RT:2008, RT1:2009, RT2:2009, TW:2021}. Thus the mathematical focus of the present paper is to establish the required Unique Continuation Property of an eigenvalue problem for the MHD equations subject to an over-determined condition. Such UCP is the primary subject of the present paper. The proof (based on pointwise Carleman-type estimate for the Laplacian) is a natural extension of those given in \cite{RT2:2009} for the Navier-Stokes equations and in \cite{TW:2021} for the Boussinesq system. As mentioned, we refer to the paper \cite{RT:2024} where the role of UCP in the case of parabolic dynamics is extensively treated. The authors' subsequent study of asymptotic turbulent suppression of the MHD system first with localized interior control in Besov spaces is given in \cite{LPT.6}, to be next followed by the localized boundary-type control case in Besov spaces.


\subsection{Controlled dynamic Magnetohydrodynamic equations}

As already noted, we wish to introduce the present unique continuation theorem in the context of uniform stabilization problem. Let, at first, $\Omega$ be an open connected bounded domain in $\mathbb{R}^d, d = 2,3$ with sufficiently smooth boundary $\Gamma = \partial \Omega$.  More specific requirements will be given below. Let $\omega$ be an arbitrarily small open smooth subset of the interior $\Omega$, $\omega \subset \Omega$, of positive measure. Let $m$ denote the characteristic function of $\omega$: $m(\omega) \equiv 1, \ m(\Omega \backslash \omega) \equiv 0$. We consider the following Magnetohydrodynamic equations perturbed by forces $f$ and $g$, and subject to the action of a pair ${u,v}$ of interior localized controls, to be described below, where $Q =  (0, \infty) \times \Omega, \ \Sigma = (0,\infty) \times \Gamma$: 
\begin{subequations}\label{1.1}
	\begin{align}
	y_t - \nu \Delta y + (y \cdot \nabla) y + \nabla \pi + \frac{1}{2} \nabla ( B \cdot B) - (B \cdot \nabla) B &= m(x)u(t,x) + f(x)   &\text{ in } Q, \label{1.1a}\\ 
	B_t + \eta \ \curl  \curl \, B + (y \cdot \nabla) B - (B \cdot \nabla)y &= m(x)v(t,x) + g(x) &\text{ in } Q, \label{1.1b}\\
	\div y = 0, \quad \text{div} \ B &= 0   &\text{ in } Q, \label{1.1c}\\
	y = 0, \ B \cdot n = 0, \ (\curl  B) \times n & = 0 &\text{ on } \Sigma, \label{2.1d}\\
	y(0,x) = y_0, \quad B(0,x) & = B_0 &\text{ on } \Omega. \label{1.1e}
	\end{align}
\end{subequations}
We note the formula
\begin{equation*}
	\curl  \curl  B = -\Delta B + \nabla \div B
\end{equation*}
so that Eqt  \eqref{1.1b} can be more conveniently rewritten as 
\begin{equation}\label{1.1b'}
	B_t - \eta \ \Delta B + (y \cdot \nabla) B - (B \cdot \nabla)y = m(x)v(t,x) + g(x) \tag{1.1b'}
\end{equation}
invoking $\div B \equiv 0$ in $Q$ from \eqref{1.1c}. Furthermore, we denote total pressure $\varrho$ in the dynamic equation as $\ds \varrho := \pi + \frac{1}{2} (B \cdot B)$ and in the static case $\ds \varrho_e := \pi_e + \frac{1}{2} (B_e \cdot B_e)$.

\subsection{Stationary Magnetohydrodynamics equations}
The following result represents our basic starting point. See \cite{AB:2020}.

\begin{thm}\label{Thm-1.1}
	Consider the following steady-state Magnetohydrodynamics equations in $\Omega$	
	\begin{subequations}\label{1.2}
		\begin{align}
		- \nu \Delta y_e + (y_e \cdot \nabla) y_e + \nabla \varrho_e - (B_e \cdot \nabla) B_e &= f(x)   &\text{in } \Omega, \label{1.2a}\\ 
		-\eta \Delta B_e + (y_e \cdot \nabla) B_e - (B_e \cdot \nabla)y_e &= g(x) &\text{in } \Omega, \label{1.2b}\\
		\div y_e = 0, \quad \div B_e &= 0   & \text{in } \Omega, \label{1.2c}\\
		y_e = 0, \ B_e \cdot n = 0, \ (\curl B_e) \times n & = 0 &\text{on } \Gamma. \label{1.2d}
		\end{align}
	\end{subequations}
	Let $1 < q < \infty$. For any $f,g \in L^q(\Omega)$, there exits a solution (not necessarily unique) $(y_e,B_e, \pi_e) \in (W^{2,q}(\Omega))^d \times (W^{2,q}(\Omega))^d \times W^{1,q}(\Omega), \ q > d$.
\end{thm}
\vspace{-20pt}
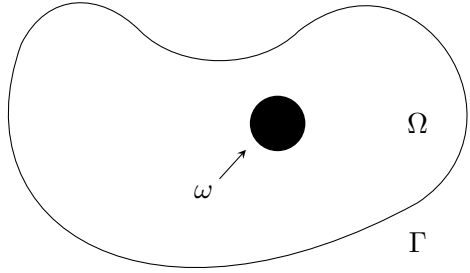
\begin{figure}[H]
	\centering
	\begin{tikzpicture}[x=10pt,y=10pt,>=stealth, scale=0.5]
	\draw
	(-15,5)
	.. controls (-20,-9) and (-5,-18) .. (15,-7)
	.. controls  (24,-1) and (15,13).. (6,6)
	.. controls (3,3) and (-3,3) .. (-6,6)
	.. controls (-9,9) and (-13,9) .. (-15,5);
	\draw (15,-1) node[scale=1] {$\Omega$};
	\draw (15,-10) node[scale=1] {$\Gamma$};
	\draw[fill, line width = 1pt, ,opacity = .25] (4.4,-1) circle (2);
	\draw (-1.2,-6.2) node[scale=1] {$\omega$};
	\draw[->]  (-0,-5.2) -- (2,-3);
	\end{tikzpicture}
	\caption{The localized interior set $\omega$}
	\label{fig:localized_pair}
\end{figure}

\subsection{Translated MHD system}
	We return to Theorem \ref{Thm-1.1} and choose an equilibrium triplet $\{y_e, B_e, \pi_e\}$ to be kept fixed throughout the analysis. Then, we translate by $\{y_e, p_e\}$ the original N-S problem (\ref{1.1}). Thus, we introduce new variables
\begin{subequations}\label{1.3}
	\begin{equation}\label{1.3a}
	z = y - y_e, \quad \mathbb{B} = B - B_e \quad p = \varrho - \varrho_e
	\end{equation}
	and obtain the translated problem given by
	\begin{eqnarray}
	z_t - \nu \Delta z + (y_e \cdot \nabla)z + (z \cdot \nabla)y_e + 
	(z \cdot \nabla) z - (B_e \cdot \nabla)\mathbb{B} - (\mathbb{B} \cdot \nabla)B_e  \hspace{4cm} \nonumber\\ - (\mathbb{B} \cdot \nabla) \mathbb{B} + \nabla p = mu  \text{ in } Q, \hspace{1cm}\\
	\mathbb{B}_t - \eta \Delta \mathbb{B} + (z \cdot \nabla) B_e + (y_e \cdot \nabla ) \mathbb{B} - (\mathbb{B} \cdot \nabla)y_e - (B_e \cdot \nabla) z  \hspace{4cm} \nonumber\\ 
	+ (z \cdot \nabla) \mathbb{B} - (\mathbb{B} \cdot \nabla) z = mv \text{ in } Q,  \hspace{1cm}\\
	\div \ z = 0, \quad \div \ \mathbb{B} = 0 \text{ in } Q, \hspace{1cm}\\
	z = 0, \ \mathbb{B} \cdot n = 0, \ (\curl \ \mathbb{B}) \times n = 0 \text{ on } \Sigma, \hspace{1cm}\\
	z(0,x) = y_0(x) - y_e(x), \quad \mathbb{B}(0,x) = B_0(x) - B_e(x) \text{ on } \Omega. \hspace{1cm}
	\end{eqnarray}
\end{subequations}

\subsection{Translated linearized MHD system}
\begin{subequations}\label{1.4}
	The translated linearized problem is
	\begin{eqnarray}
	w_t - \nu \Delta w + (y_e \cdot \nabla)w + (w \cdot \nabla)y_e - (\mathbb{W} \cdot \nabla)B_e - (B_e \cdot \nabla)\mathbb{W} + \nabla p = mu  \text{ in } Q,\\
	\mathbb{W}_t - \eta \Delta \mathbb{W} + (w \cdot \nabla) B_e - (B_e \cdot \nabla) w + (y_e \cdot \nabla ) \mathbb{W} - (\mathbb{W} \cdot \nabla)y_e = mv \text{ in } Q,\\ 
	\div \ w = 0, \quad \div \ \mathbb{W} = 0 \text{ in } Q, \label{1.4c}\\
	w = 0, \ \mathbb{W} \cdot n = 0, \ (\curl \ \mathbb{W}) \times n = 0 \text{ on } \Sigma, \label{1.4d}\\
	w(0,x) = y_0 - y_e, \quad \mathbb{W}(0,x) = B_0 - B_e \text{ on } \Omega.
	\end{eqnarray}
\end{subequations}

\subsection{The required unique continuation theorem to uniformly stabilize the linear problem \eqref{1.4} by localized finitely many static feedback controls in \cite{LPT.6}}

As described in the introduction, solution to the desired uniform stabilization problem of the original, non-linear problem (\hyperref[1.1]{1.1a-e}) in the vicinity of an unstable equilibrium solution $\{y_e,B_e\}$ is given in \cite{LPT.6}. A critical preliminary step is the uniform stabilization of the linear problem (\hyperref[1.4]{1.4a-e}). To achieve it, the following Unique Continuation result it critical. Its implication on the sought-after uniform stabilization is shown in \cite{LPT.6}.

\begin{thm}\label{Thm-UCP-1}(UCP, direct problem) Let $\omega$ be an arbitrary open, connected smooth subset of $\Omega$, thus of positive measure, Fig. \ref{fig:localized_pair}. Let $\{ \phi, \xi, p \} \in (W^{2,q}(\Omega))^d \times (W^{2,q}(\Omega))^d \times W^{1,q}(\Omega), \ q > d$, solve the original eigenvalue problem
	\begin{subequations}\label{1.5}
		\begin{eqnarray}
		- \nu \Delta \phi + (y_e \cdot \nabla) \phi + (\phi \cdot \nabla)y_e - (B_e \cdot \nabla)\xi - (\xi \cdot \nabla)B_e  + \nabla p = \lambda \phi  \text{ in } \Omega, \label{1.5a}\\
		 - \eta \Delta \xi + (\phi \cdot \nabla) B_e - (B_e \cdot \nabla) \phi + (y_e \cdot \nabla ) \xi - (\xi \cdot \nabla)y_e  = \lambda \xi \text{ in } \Omega, \label{1.5b}\\ 
		\div \ \phi = 0, \quad \div \ \xi = 0 \text{ in } \Omega, \label{1.5c}\\
		\phi = 0, \ \xi \cdot n = 0, \ (\curl \xi) \times n = 0 \text{ on } \Gamma. \label{1.5d}
		\end{eqnarray}
	\end{subequations}
	along with the over-determined condition
	\begin{equation}\label{1.6}
		\phi \equiv 0, \quad \xi \equiv 0 \quad \text{ in } \omega.
	\end{equation}
	Then 
	\begin{equation}
		\phi \equiv 0, \quad \xi \equiv 0, \quad p \equiv const \text{ in } \Omega. \qed \label{1.7}
	\end{equation}
\end{thm}

\nin In line with the literature of Navier-Stokes equations, it will be convenient to introduce the following first order operators
\begin{subequations}\label{2.6}
	\begin{align}
	\calL_1 \phi &= \calL^+_{y_e}\phi = (y_e \cdot \nabla) \phi + (\phi \cdot \nabla)y_e, \label{2.6a}\\
	\calL_2 \xi &= \calL^+_{B_e}\xi = (B_e \cdot \nabla)\xi + (\xi \cdot \nabla)B_e. \label{2.6b}
	\end{align}
\end{subequations}
\vspace{-20pt}
\begin{subequations}\label{1.9}
	\begin{align}
	\calM_1 \xi &= \calL^-_{y_e}\xi = (y_e \cdot \nabla ) \xi - (\xi \cdot \nabla)y_e, \label{1.9a}\\
	\calM_2 \phi &= \calL^-_{B_e}\phi = (B_e \cdot \nabla) \phi - (\phi \cdot \nabla) B_e, \label{1.9b}
	\end{align}
\end{subequations}

\nin $\calL^+_{y_e}$ and $\calL^+_{B_e}$ being the Oseen operators for $y_e$ and $B_e$ respectively. For convenience in the analysis below, we re-write the $\phi$-equation \eqref{1.5a} and the $\xi$-equation \eqref{1.5b}, by use of \eqref{2.6}, \eqref{1.9} as 
\begin{subequations}
	\begin{align}
		-\nu \Delta \phi + \calL_1 \phi - \calL_2 \xi + \nabla p &= \lambda \phi \text{ in } \Omega, \tag{1.5a'} \label{1.5a'}\\
		-\eta \Delta \xi + \calM_1 \xi - \calM_2 \phi  &= \lambda \xi \text{ in } \Omega. \tag{1.5b'}\label{1.5b'}
	\end{align}
\end{subequations}

\subsection{Literature}
A comparison between the results on uniform stabilization of the MHD problem (\hyperref[1.1]{1.1a-e}) in Besov spaces \cite{LPT.6} and past results in the literature (all in Hilbert spaces) is given in \cite[Section 1.5]{LPT.6}. In particular, \cite{LPT.6} requires a new maximal $L^p$-regularity \cite{LPT.5} (see also \cite{LPT.7}) in the Besov setting, while by contrast in the Hilbert setting of \cite{Lef:2011}, only analyticity is needed (which is equivalent to maximal $L^2$-regularity \cite{DS:1964}), a less challenging task. Paper \cite{LPT.6} \uline{constructs explicitly} the finite-dimensional, stabilizing feedback controllers, and of minimal numbers, while \cite{Lef:2011} simply asserts the existence of a non-constructed feedback operator with finite dimensional range of unspecified dimension. The finite dimensional decomposition approach introduced in \cite{RT:1975}, and followed in both \cite{Lef:2011} and \cite{LPT.6}, requires a unique continuation property (UCP) result to assert Kalman algebraic, finite rank condition of the finite-dimensional unstable component of the overall system. To achieve this end, \cite{Lef:2011} establishes a UCP for a \uline{dynamic} coupled problem, by virtue of Carleman-type inequalities for \uline{parabolic} equations, coupled with elliptic estimates. In contrast, we only need to establish a UCP for a static eigenvalue problem [\eqref{1.5}, \eqref{1.6} of the present paper] still by Carleman-type estimates \cite{LPT.6}, a much more direct task. The original Carleman estimates, characterized by a suitable exponential weight function, were introduced in \cite{Car} in 1939 to establish the uniqueness of solutions for a PDE in two variables. The method of Carleman estimates was subsequently extended to the study of uniqueness in inverse problems, as first introduced in \cite{B.1, BK.1}. For further developments and applications in this context, see also \cite{K.1}. The subject has since grown substantially, with a vast and rich body of literature.

\section{Proof of Theorem \ref{Thm-UCP-1}}

	\nin \textbf{Step 0}. Without loss of generality, we may normalize the constants $\nu = \eta \equiv 1$.
	We can rewrite Equations \eqref{1.5a}-\eqref{1.5b} combined as in \eqref{2.2a} below, along with \eqref{1.5c} and \eqref{1.5d}, and the over-determination \eqref{1.6}
	\begin{subequations}\label{2.2}
		\begin{empheq}[left=\empheqlbrace]{gather}
			(-\Delta) \bbm \phi \\ \xi \ebm + (y_e \cdot \nabla) \bbm \phi \\ \xi \ebm + (\phi \cdot \nabla) \bbm y_e \\ B_e \ebm \hspace{6cm} \nonumber\\
			 \hspace{1.5cm} - (\xi \cdot \nabla) \bbm B_e \\ y_e \ebm - (B_e \cdot \nabla) \bbm \xi \\ \phi \ebm + \bbm \nabla p \\ 0 \ebm = \lambda \bbm \phi \\ \xi \ebm \text{ in } \Omega, \label{2.2a}\\[1mm]
			\div \bbm \phi \\ \xi \ebm \equiv \bbm 0 \\ 0 \ebm \text{ in } \Omega, \text{ and } \bbm \phi \\ \xi \ebm = \bbm 0 \\ 0 \ebm \text{ in } \omega. \label{2.2b} \hspace{2cm}
		\end{empheq}
	\end{subequations}
\nin Now we define a coordinate switching operator $\mathcal{S}$ such that $\ds \bbm \xi \\ \phi \ebm \mapsto \bbm \phi \\ \xi \ebm$ which is clearly bounded and continuous. Hence rewrite the above equation as 
\begin{subequations}
	\begin{empheq}[left=\empheqlbrace]{gather}
	(-\Delta) \bbm \phi \\ \xi \ebm + (y_e \cdot \nabla) \bbm \phi \\ \xi \ebm + (\phi \cdot \nabla) \bbm y_e \\ B_e \ebm \hspace{6cm} \nonumber\\
	\hspace{1.5cm} - (\xi \cdot \nabla) \ \calS \bbm y_e \\ B_e \ebm - (B_e \cdot \nabla) \ \calS \bbm \phi \\ \xi \ebm + \bbm \nabla p \\ 0 \ebm = \lambda \bbm \phi \\ \xi \ebm \text{ in } \Omega,\\[1mm]
	\div \bbm \phi \\ \xi \ebm \equiv \bbm 0 \\ 0 \ebm \text{ in } \Omega, \text{ and } \bbm \phi \\ \xi \ebm = \bbm 0 \\ 0 \ebm \text{ in } \omega. \hspace{2cm}
	\end{empheq}
\end{subequations}

\nin \textbf{Case 1}: We write initially the proof for the case where $\omega$ is at a positive distance from $\partial \Omega$: dist $(\partial \Omega, \partial \omega) > 0$. (Figs. \ref{fig:1}, \ref{fig:2}).\\
\nin \textbf{Step 1}. Henceforth, we introduce the state variable $\ds u = \bbm \phi \\ \xi \ebm$. Since $u = \{ \phi,\ \xi \} \equiv 0$ in $\omega$ by \eqref{2.2b}, then \eqref{2.2a}, yields $\ds \nabla p \equiv 0$ in $\omega$, hence $p = $ const in $\omega$. We may, and will, take $p \equiv 0$ in $\omega$, as $p$ is only identified up to a constant. Then we have:
\begin{equation}    \label{E4-2-1}
u|_{\partial\omega} \equiv 0; \quad \frac{\partial u}{\partial\nu}\bigg|_{\partial\omega} \equiv 0; \quad
p|_{\partial \omega} \equiv 0; \quad \frac{\partial p}{\partial\nu}\bigg|_{\partial\omega} \equiv 0,
\end{equation}
where $\ds \frac{\partial}{\partial\nu}$ denotes the normal derivative ($\nu =$ unit inward normal vector with respect to $\omega$).\\

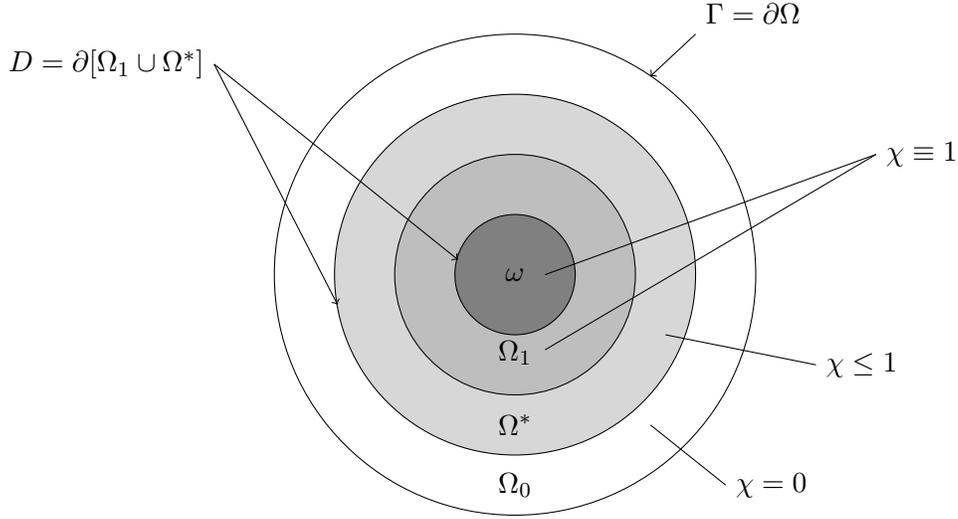
\begin{figure}[H]
	\centering
	\begin{tikzpicture}[scale=0.4]
	
	\draw (0,-1) node {$\chi \equiv 1$};
	\draw  (0,0) circle [radius = 8]; 
	\draw [fill={rgb, 255:red, 215; green, 215; blue, 215 }  ,fill opacity=1 ] (0,0) circle [radius = 6];
	\draw [fill={rgb, 255:red, 190; green, 190; blue, 190},fill opacity=1 ] (0,0) circle [radius = 4];
	\draw [fill={rgb, 255:red, 128; green, 128; blue, 128 }, ,fill opacity=1 ] (0,0) node {$\omega$} circle [radius = 2];

	\draw (0,-2.6) node {$\Omega_1$};
	\draw (0,-5) node {$\Omega^*$};
	\draw (0,-7) node {$\Omega_0$};
	
	\draw (10,-3) node[anchor = west] {$\chi \leq 1$} -- (5,-2);
	
	\draw (12,4) node[anchor = west] {$\chi \equiv 1$} -- (1,0);
	\draw (12,4) -- (1,-2.5);
	
	\draw (7,-7) node[anchor = west] {$\chi = 0$} -- (4.5,-5);
	\draw (-10,-3) node[anchor = east] {\qquad \quad};
	\draw[->] (-10,7) node[anchor = east] {$D = \partial[\Omega_1 \cup \Omega^*]$} -- (-1.9,0.5);
	\draw[->] (-10,7)  -- (-5.9,-1);
	\draw[->] (6,8) node[anchor = south west] {$\Gamma = \partial \Omega$} -- (4.5,6.6);
	\end{tikzpicture}
	\caption{Case 1: $G = \Omega_1 \cup \Omega^*$}
	\label{fig:1}
\end{figure}
\begin{figure}[H]
	\centering
	
	\tikzset{every picture/.style={line width=0.75pt}} 
	
	\begin{tikzpicture}[x=0.75pt,y=0.75pt,yscale=-1,xscale=0.8]
	
	\draw  [draw opacity=0][fill={rgb, 255:red, 215; green, 215; blue, 215 }  ,fill opacity=1 ][line width=0.75]  (459,154.8) -- (461,166.8) -- (462,177.8) -- (463,189.8) -- (465,201.8) -- (470,212.8) -- (478,221.8) -- (488,228.8) -- (498,232.8) -- (510,235.8) -- (523,237.8) -- (542,239.8) -- (430,240) -- (430,138.95) -- (442,140.8) -- (453,145.8) -- cycle ;
	\draw  [draw opacity=0][fill={rgb, 255:red, 215; green, 215; blue, 215 }  ,fill opacity=1 ][dash pattern={on 0.84pt off 2.51pt}][line width=0.75]  (181,161.4) -- (179,172.4) -- (179,184.4) -- (177,196.4) -- (174,209.4) -- (169,218.4) -- (163,225.4) -- (152,231.4) -- (142,234.4) -- (131,237.4) -- (118.81,238.94) -- (100,240) -- (211.44,240) -- (211.44,138.95) -- (196,141.4) -- (184,149.4) -- cycle ;
	\draw  [draw opacity=0][fill={rgb, 255:red, 190; green, 190; blue, 190 }  ,fill opacity=1 ][dash pattern={on 0.84pt off 2.51pt}] (390,140) -- (430,140) -- (430,240) -- (390,240) -- cycle ;
	\draw  [fill={rgb, 255:red, 190; green, 190; blue, 190 }  ,fill opacity=1 ] (210,140) -- (320,140) -- (320,240) -- (210,240) -- cycle ;
	\draw  [fill={rgb, 255:red, 128; green, 128; blue, 128 }  ,fill opacity=1 ][dash pattern={on 4.5pt off 4.5pt}] (320,140) -- (390,140) -- (390,240) -- (320,240) -- cycle ;
	\draw    (20,240) -- (630,240) ;
	\draw    (100,240) .. controls (232,240.4) and (140,140.4) .. (210,140) ;
	\draw    (210,140) -- (430,140) ;
	\draw    (320,140) -- (320,240) ;
	\draw    (330,280) -- (630,280) ;
	\draw [shift={(630,280)}, rotate = 180] [color={rgb, 255:red, 0; green, 0; blue, 0 }  ][line width=0.75]    (0,5.59) -- (0,-5.59)   ;
	\draw    (310,280) -- (20,280) ;
	\draw [shift={(20,280)}, rotate = 360] [color={rgb, 255:red, 0; green, 0; blue, 0 }  ][line width=0.75]    (0,5.59) -- (0,-5.59)   ;
	\draw    (70,260) -- (100,260) ;
	\draw [shift={(100,260)}, rotate = 180] [color={rgb, 255:red, 0; green, 0; blue, 0 }  ][line width=0.75]    (0,5.59) -- (0,-5.59)   ;
	\draw    (50,260) -- (20,260) ;
	\draw [shift={(20,260)}, rotate = 360] [color={rgb, 255:red, 0; green, 0; blue, 0 }  ][line width=0.75]    (0,5.59) -- (0,-5.59)   ;
	\draw    (600,260) -- (630,260) ;
	\draw [shift={(630,260)}, rotate = 180] [color={rgb, 255:red, 0; green, 0; blue, 0 }  ][line width=0.75]    (0,5.59) -- (0,-5.59)   ;
	\draw    (580,260) -- (550,260) ;
	\draw [shift={(550,260)}, rotate = 360] [color={rgb, 255:red, 0; green, 0; blue, 0 }  ][line width=0.75]    (0,5.59) -- (0,-5.59)   ;
	\draw    (170,260) -- (210,260) ;
	\draw [shift={(210,260)}, rotate = 180] [color={rgb, 255:red, 0; green, 0; blue, 0 }  ][line width=0.75]    (0,5.59) -- (0,-5.59)   ;
	\draw    (150,260) -- (100,260) ;
	\draw [shift={(100,260)}, rotate = 360] [color={rgb, 255:red, 0; green, 0; blue, 0 }  ][line width=0.75]    (0,5.59) -- (0,-5.59)   ;
	\draw    (510,260) -- (550,260) ;
	\draw [shift={(550,260)}, rotate = 180] [color={rgb, 255:red, 0; green, 0; blue, 0 }  ][line width=0.75]    (0,5.59) -- (0,-5.59)   ;
	\draw    (490,260) -- (430,260) ;
	\draw [shift={(430,260)}, rotate = 360] [color={rgb, 255:red, 0; green, 0; blue, 0 }  ][line width=0.75]    (0,5.59) -- (0,-5.59)   ;
	\draw    (280,260) -- (320,260) ;
	\draw [shift={(320,260)}, rotate = 180] [color={rgb, 255:red, 0; green, 0; blue, 0 }  ][line width=0.75]    (0,5.59) -- (0,-5.59)   ;
	\draw    (260,260) -- (210,260) ;
	\draw [shift={(210,260)}, rotate = 360] [color={rgb, 255:red, 0; green, 0; blue, 0 }  ][line width=0.75]    (0,5.59) -- (0,-5.59)   ;
	\draw    (420,260) -- (430,260) ;
	\draw [shift={(430,260)}, rotate = 180] [color={rgb, 255:red, 0; green, 0; blue, 0 }  ][line width=0.75]    (0,5.59) -- (0,-5.59)   ;
	\draw    (400,260) -- (390,260) ;
	\draw [shift={(390,260)}, rotate = 360] [color={rgb, 255:red, 0; green, 0; blue, 0 }  ][line width=0.75]    (0,5.59) -- (0,-5.59)   ;
	\draw    (364.55,260) -- (390,260) ;
	\draw [shift={(390,260)}, rotate = 180] [color={rgb, 255:red, 0; green, 0; blue, 0 }  ][line width=0.75]    (0,5.59) -- (0,-5.59)   ;
	\draw    (350,260) -- (320,260) ;
	\draw [shift={(320,260)}, rotate = 360] [color={rgb, 255:red, 0; green, 0; blue, 0 }  ][line width=0.75]    (0,5.59) -- (0,-5.59)   ;
	\draw    (430,140) .. controls (500,139.8) and (408,239.8) .. (550,239.8) ;
	\draw    (430,140) -- (430,240) ;
	
	\draw (311,269) node [anchor=north west][inner sep=0.75pt]    {$\Omega $};
	\draw (49,249) node [anchor=north west][inner sep=0.75pt]  [font=\footnotesize]  {$\Omega _{0}$};
	\draw (581,249) node [anchor=north west][inner sep=0.75pt]  [font=\small]  {$\Omega _{0}$};
	\draw (149,251) node [anchor=north west][inner sep=0.75pt]  [font=\footnotesize]  {$\Omega ^{*}$};
	\draw (489,251) node [anchor=north west][inner sep=0.75pt]  [font=\footnotesize]  {$\Omega ^{*}$};
	\draw (259,251) node [anchor=north west][inner sep=0.75pt]  [font=\footnotesize]  {$\Omega _{1}$};
	\draw (399,251) node [anchor=north west][inner sep=0.75pt]  [font=\footnotesize]  {$\Omega _{1}$};
	\draw (349,249) node [anchor=north west][inner sep=0.75pt]  [font=\footnotesize]  {$\omega $};
	\draw (251,109) node [anchor=north west][inner sep=0.75pt]    {$\chi \ \equiv 1$};

	\end{tikzpicture}
	\caption{Case 1: the cut-off function $\chi$}
	\label{fig:2}
\end{figure}
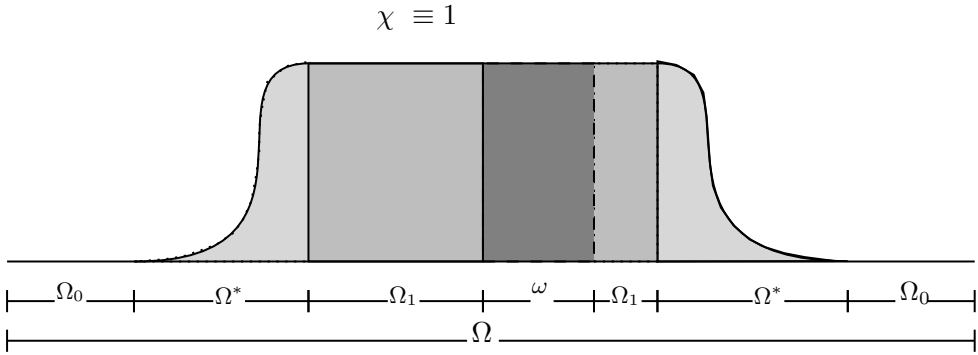

\noindent \textbf{Step 2}. \textit{The cut-off function $\chi$.}
Let $\chi$ be a smooth, non-negative, cut-off function defined as follows:
\begin{equation}    \label{E4-2-2}
\chi \equiv
\begin{cases}
1 & \mbox{in } \Omega_1 \cup \omega \\[2mm]
0 & \mbox{in } \Omega_0
\end{cases};
\quad \mbox{ supp } \chi \subset [\Omega_1 \cup \omega \cup \Omega^*],
\end{equation}
while monotonically decreasing from 1 to 0 in $\Omega^*$, with $\chi \equiv 0$ also in a small layer of $\Omega^*$ bordering $\Omega_0$ (Fig. \ref{fig:1} and Fig. \ref{fig:2}). Here,

(i) $\Omega_1$ is a smooth subdomain of $\Omega$ which surrounds and borders $\omega$ (Fig. \ref{fig:1});

(ii) in turn, $\Omega^*$ is a smooth subdomain of $\Omega$ which surrounds and borders $\Omega_1$ (Fig. \ref{fig:1});

(iii) in turn, $\Omega_0$ is a smooth subdomain of $\Omega: \Omega_0 \equiv \Omega \setminus \{\omega \cup \Omega_1 \cup \Omega^*\}$.\\

\noindent \textbf{Step 3}. \textit{The $(\chi \phi)$-problem.}
Multiply the $\phi$-equation re-written as in \eqref{1.5a'} by $\chi$ and obtain:
\begin{equation}\label{2.5}
	(- \Delta) (\chi \phi) + \calL_1(\chi \phi) - \calL_2(\chi \xi) + \nabla (\chi p) = \lambda (\chi \phi) + F_\chi^{1,0,0}(\phi,\xi,p) \mbox{ in } \Omega,
\end{equation}
or recalling (\hyperref[2.6]{1.8a-b})
\begin{multline}\label{E4-2-3}
	(- \Delta) (\chi \phi) + (y_e \cdot \nabla) (\chi \phi) + ((\chi \phi) \cdot \nabla)y_e\\ - (B_e \cdot \nabla)(\chi \xi) - ((\chi \xi) \cdot \nabla)B_e + \nabla (\chi p) = \lambda (\chi \phi) + F_\chi^{1,0,0}(\phi,\xi,p) \mbox{ in } \Omega,
\end{multline}
with forcing term expressed in terms of the resulting commutators
\begin{subequations}    \label{E4-2-4}
	\begin{align}
	F_\chi \equiv F^{1,0,0}_\chi(\phi,\xi,p) 
	& \equiv [\chi,\Delta]\phi + [\calL_1,\chi]\phi - [\calL_2,\chi]\xi + [\nabla,\chi]p           \label{E4-2-4a}
	\\[2mm]
	& =  \mbox{first order in } \phi; \mbox{ zero order in } p \text{ and } \xi; \label{E4-2-4b}\\
	& \mbox{ supp } F^{1,0,0}_\chi \subset \Omega^*.    \label{E4-2-4c}
	\end{align}
\end{subequations}
Notice that \eqref{E4-2-4c} holds true since $\chi \equiv 1$ on $\Omega_1$, on $\omega$, and on a small layer within $\Omega^*$, so that on the union taken over these three sets we have that $F_{\chi} \equiv 0$.
We recall that the commutator $[\chi, \Delta]$ is defined by $[\chi, \Delta]\phi = \chi \Delta \phi - \Delta(\chi \phi)$. Thanks to the Leibniz formula, $[\chi, \Delta]$ is a linear combination of derivatives of order $\geq 1$ of $\chi$ multiplied by derivatives of order $\leq 1$ of $\phi$. Consequently, the commutator $[\chi, \Delta]$ is of order $0+2-1=1$; the commutator $[\chi, \calL_i]$ is of order $0+1-1 = 0,\ i = 1,2$; the commutator $[\nabla, \chi]$ is of order $1+0-1=0$.\\

\nin \textbf{Step 4}. \textit{The $(\chi \xi)$-problem}. 
Next, we multiply the $\xi$-equation re-written as in \eqref{1.5b'} by $\chi$ and obtain
\begin{equation}
	(- \Delta)(\chi \xi) + \calM_1(\chi \xi) - \calM_2(\chi \phi) = \lambda (\chi \xi) + G^{1,0}_{\chi}(\xi,\phi) \quad \mbox{ in } \Omega,
\end{equation}
or
\begin{multline}
(- \Delta)(\chi \xi) + (y_e \cdot \nabla ) (\chi\xi) - ((\chi\xi) \cdot \nabla)y_e \\ - (B_e \cdot \nabla) (\chi\phi) + ((\chi\phi) \cdot \nabla) B_e = \lambda (\chi \xi) + G^{1,0}_{\chi}(\xi,\phi) \quad \mbox{ in } \Omega, \label{6E2-7}
\end{multline}
with forcing term expressed in terms of the resulting commutators
\begin{subequations} \label{6E2-9}
	\begin{align}
		G_{\chi} = G^{1,0}_{\chi}(\xi,\phi) & \equiv [\chi,\Delta]\xi + [\calM_1,\chi]\xi + [\calM_2,\chi]\phi, \label{6E2-9a}\\
		& = \text{first order in } \xi; \text{ zero order in } \phi; \text{ supp } G^{1,0}_{\chi} \subset \Omega^*. \label{6E2-9b}
	\end{align}
\end{subequations}
Notice that \eqref{6E2-9b} holds true, since as in the case for \eqref{E4-2-4c}, we have that $G_{\chi} \equiv 0$ on $\omega \cup \Omega_1 \cup \mbox{ [a small layer within $\Omega^*$]}$, since $\chi \equiv 1$ on such union.\\

\nin \textbf{Step 5}. \textit{The $(\chi u)$-problem $\chi u = \{ \chi \phi, \chi \xi \}$.} We combine Step 3 and Step 4 and obtain:
\begin{multline}\label{2.12}
	(-\Delta) \bpm \chi \bbm \phi \\ \xi \ebm \epm + (y_e \cdot \nabla) \bpm \chi \bbm \phi \\ \xi \ebm \epm + ((\chi \phi) \cdot \nabla) \bbm y_e \\ B_e \ebm \\ - ((\chi \xi) \cdot \nabla) \bbm B_e \\ y_e \ebm - (B_e \cdot \nabla) \bpm \chi \bbm \xi \\ \phi \ebm \epm + \bbm \nabla (\chi p) \\ 0 \ebm = \lambda \bpm \chi \bbm \phi \\ \xi \ebm \epm + \ds \bbm  F^{1,0,0}_{\chi} \\[1mm] G^{1,0}_{\chi} \ebm \text{ in } \Omega.
\end{multline}
Moreover, let (Fig. \ref{fig:1})
\begin{equation}\label{E4-2-5}
D = \partial \omega \cup \{\mbox{external boundary of } \Omega^*\} = \partial[\Omega_1 \cup \Omega^*].
\end{equation}
Since $\chi \equiv 0$ on $\Omega_0$ and in a small layer of $\Omega^*$ bordering $\Omega_0$, then $(\chi u) = \{ (\chi \phi), (\chi \xi) \}$ and $(\chi p)$ have zero Cauchy data on the [external boundary of $\Omega^*$] = [interior boundary of $\Omega_0$].  Moreover, since $u \equiv 0$ in $\omega$ and $p \equiv 0$ in $\omega$ by Step 1, then $(\chi u)$ and $(\chi p)$ have zero Cauchy data on $\partial \omega$. Thus:
\begin{equation}\label{E4-2-6}
(\chi u)|_D \equiv 0; \ \frac{\partial(\chi u)}{\partial\nu}\bigg|_D \equiv 0; \ (\chi p)|_{\partial\omega} \equiv 0; \ \frac{\partial(\chi p)}{\partial\nu}\bigg|_{\partial\omega} \equiv 0,
\end{equation}
where $\nu$ denotes a unit normal vector outward with respect to $[\Omega^* \cup \Omega_1]$ (Fig. \ref{fig:1}).\\

\noindent \textbf{Step 6}. \textit{A pointwise Carleman estimate.}
We shall invoke the following pointwise Carleman estimate for the Laplacian from \cite[Corollary 4.2, Eqn.\,(4.15), p.\,73]{L-T-Z.2}.\\

\begin{theorem}
	\label{thm4-2-1}
	The following pointwise estimate holds true at each point $x$ of a bounded domain $G$ in $\mathbb{R}^d$ for an $H^2$-function $w$, where $\epsilon > 0$ and $0 < \delta_0 < 1$ are arbitrary
	\begin{multline} \label{E4-2-7}
	\delta_0 \left[ 2 \rho \tau - \frac{\epsilon}{2}\right] e^{2\tau \psi(x)} |\nabla w(x)|^2+ [4 \rho k^2\tau^3(1-\delta_0) + \calO(\tau^2)] e^{2\tau \psi(x)} |w(x)|^2   \\
	\leq  
	\left(1 + \frac{1}{\epsilon}\right) e^{2\tau\psi(x)} |\Delta w(x)|^2 + \div V_w(x), \quad x \in G.
	\end{multline}
	Here: $\psi(x)$ is any strictly convex function over $G$, with no critical points in $\overline{G}$, to be chosen below in Step 11 when $G = \Omega_1 \cup \Omega^*$; $\rho > 0$ is a constant, defined by $\calH_\psi(x) \geq \rho I$, $x \in \overline{G}$, where $\calH_\psi$ denotes the (symmetric) Hessian matrix of $\psi(x)$ \cite[Eqn.\,(1.1.6), p.\,45]{L-T-Z.2}; $k > 0$ is a constant, defined by: $\mbox{inf }|\nabla \psi(x)| = k > 0$, where the inf is taken over $G$ \cite[Eqn.\,(1.1.7), p.\,45]{L-T-Z.2}; and $\tau$ is a free positive parameter, to be chosen sufficiently large. For what follows, it is not critical to recall what $\div V_w(x)$ is, only that, via the divergence theorem, we have
	\begin{equation}         \label{E4-2-8}
	\int_G \div V_w(x) dx = \int_{\partial G} V_w(x) \cdot \nu \, d\sigma = 0,
	\end{equation}
	whenever the {\em Cauchy data} of $w$ vanish on its boundary $\partial G$: $w|_{\partial G} \equiv 0; \ \nabla w|_{\partial G} \equiv 0$. In \eqref{E4-2-8}, $\nu$ is a unit normal vector outward with respect to $G$.
\end{theorem}

\noindent \textbf{Step 7}. \textit{Pointwise Carleman estimates for $(\chi u)$.}
Next, we apply estimate \eqref{E4-2-7} with $w = (\chi u)$ solution of \eqref{2.12}. For definiteness, we select $\delta_0 = \frac{1}{2}$, $\epsilon = \frac{1}{2}$. We obtain
\begin{multline}\label{E4-2-9}
\left[ \rho \tau - \frac{1}{8}\right] e^{2\tau \psi(x)}
|\nabla (\chi u)(x)|^2+ [2 \rho k^2\tau^3+ \calO(\tau^2)] e^{2\tau \psi(x)} |(\chi u)(x)|^2\\[3mm]
\leq 3 e^{2\tau \psi(x)} |\Delta (\chi u)(x)|^2 + \div V_{(\chi u)}(x), \quad x \in G.
\end{multline}
Next, we integrate \eqref{E4-2-9} over the domain $G \equiv [\Omega_1 \cup \Omega^*]$ (Fig. \ref{fig:1}), thus obtaining
\begin{multline}\label{E4-2-10}
\left[ \rho\tau - \frac{1}{8}\right] \int_{\Omega_{1}\cup \Omega^*} e^{2\tau \psi(x)} |\nabla (\chi u)(x)|^2dx
+ [2 \rho k^2 \tau^3+ \calO(\tau^2)]\int_{\Omega_{1} \cup \Omega^*}e^{2\tau \psi(x)} |(\chi u)(x)|^2 dx\\
\leq 3 \int_{\Omega_{1}\cup \Omega^*} e^{2\tau \psi(x)} |\Delta (\chi u)(x)|^2dx + \int_{\partial[\Omega_{1}\cup \Omega^*]} \cancel{V_{(\chi u)}(x)} \cdot \nu \, dD,
\end{multline}
where, on the RHS of \eqref{E4-2-10}, the boundary integral over $D \equiv \partial[\Omega_1 \cup \Omega^*] =$ the boundary of $\Omega_1 \cup \Omega^*$, see \eqref{E4-2-5} and Fig. \ref{fig:1}, vanishes in view of \eqref{E4-2-8} with $w = (\chi u)$ having null Cauchy data on $D$, by virtue of (the LHS of) \eqref{E4-2-6}.\\

\noindent \textbf{Step 8}. \textit{Bound on the RHS of \eqref{E4-2-10}.} Here, we estimate the RHS of \eqref{E4-2-10}. Returning to the $(\chi u)$-problem \eqref{2.12}, we rewrite it over $G \equiv [\Omega_1 \cup \Omega^*]$ as

\begin{multline}\label{6E2-17}
\Delta \bpm \chi \bbm \phi \\ \xi \ebm \epm = (y_e \cdot \nabla) \bpm \chi \bbm \phi \\ \xi \ebm \epm + ((\chi \phi) \cdot \nabla) \bbm y_e \\ B_e \ebm \\ - ((\chi \xi) \cdot \nabla) \bbm B_e \\ y_e \ebm - (B_e \cdot \nabla) \bpm \chi \bbm \xi \\ \phi \ebm \epm + \bbm \nabla (\chi p) \\ 0 \ebm - \lambda \bpm \chi \bbm \phi \\ \xi \ebm \epm - \bbm F^{1,0,0}_{\chi} \\ G^{1,0}_{\chi} \ebm
\end{multline}
and multiply across by $e^{\tau \psi(x)}$ to get
\begin{multline}\label{6E2-18}
e^{\tau \psi(x)} \Delta \bpm \chi \bbm \phi \\ \xi \ebm \epm = (e^{\tau \psi(x)} y_e \cdot \nabla) \bpm \chi \bbm \phi \\ \xi \ebm \epm + (e^{\tau \psi(x)}(\chi \phi) \cdot \nabla) \bbm y_e \\ B_e \ebm - (e^{\tau \psi(x)}(\chi \xi) \cdot \nabla) \bbm B_e \\ y_e \ebm \\ - ( e^{\tau \psi(x)}B_e \cdot \nabla) \bpm \chi \bbm \xi \\ \phi \ebm \epm + e^{\tau \psi(x)} \bbm \nabla (\chi p) \\ 0 \ebm - \lambda e^{\tau \psi(x)} \bpm \chi \bbm \phi \\ \xi \ebm \epm - e^{\tau \psi(x)} \bbm F^{1,0,0}_{\chi} \\ G^{1,0}_{\chi} \ebm.
\end{multline}
Recalling $(y_e, B_e) \in (W^{2,q}(\Omega))^d \times (W^{2,q}(\Omega))^d$ by Theorem \ref{Thm-1.1}, as well as the embedding $W^{1,q}(\Omega) \hookrightarrow C(\overline{\Omega})$ for $q > d$, \cite[p 97, for $\Omega$ having cone property]{A:1975} \cite[p. 79, requiring $C^1$-boundary]{SK:1989}, we have $|\nabla y_e(x)| + |\nabla B_e(x)| \leq C_{y_e, B_e}, \ x \in \Omega$, for $q > d$, as assumed. In view of this, we return to \eqref{6E2-18} and obtain
\begin{multline}\label{6E2-19}
e^{2\tau \psi(x)} \left| \Delta \bpm \chi \bbm \phi \\ \xi \ebm \epm (x) \right|^2 
\leq c_e e^{2\tau \psi(x)} \left\{ \left|\nabla \bpm \chi \bbm \phi \\ \xi \ebm \epm(x) \right|^2 + 
\abs{(\chi \phi)(x)}^2 + \abs{(\chi \xi)(x)}^2 \right\} \\
+ \ c_{\lambda} e^{2\tau \psi(x)} \left|\left(\chi \bbm \phi \\ \xi \ebm \right)(x)\right|^2 + e^{2\tau \psi(x)} \abs{\nabla (\chi p)}^2 +  e^{2\tau \psi(x)} \left| \bbm F^{1,0,0}_{\chi}(\phi,\xi , p)(x)\\[1mm]
G^{1,0}_{\chi}(\xi, \phi)(x) \ebm \right|^2, \\
x \in G,
\end{multline}
$c_e = $ a constant depending on $y_e$ and $B_e$, $c_{\lambda} = |\lambda|^2 + 1$. Thus, integrating \eqref{6E2-19} over $G \equiv [\Omega_1 \cup \Omega^*]$ as required by \eqref{E4-2-10} yields with $u = \bbm \phi \\ \xi \ebm$
\begin{multline}        \label{6E2-20}
\int_{\Omega_{1} \cup \Omega^*} e^{2\tau \psi(x)} |\Delta(\chi u)(x)|^2 dx
\leq  C_{\lambda, e} 
\int_{\Omega_{1} \cup \Omega^*} e^{2\tau \psi(x)} \left[ \abs{\nabla (\chi u)(x)}^2+ \abs{(\chi u)(x)}^2 \right]dx\\
+  \int_{\Omega_{1} \cup \Omega^*} e^{2\tau \psi(x)} \abs{\nabla (\chi p)(x)}^2 dx
+  \int_{\Omega_{1} \cup \Omega^*} e^{2\tau \psi(x)} \left| \bbm
F^{1,0,0}_{\chi}(\phi, \xi, p)(x)\\[1mm]
G^{1,0}_{\chi}(\xi, \phi)(x)
\ebm \right|^2 dx,
\end{multline}
$C_{\lambda, e} = $ a constant depending on $\lambda, y_e$ and $B_e$.\\

\nin We now recall from \eqref{E4-2-10} and \eqref{6E2-9} that $F^{1,0,0}_{\chi}(\phi, \xi, p)$ is an operator which is first order in $\phi$ and zero order in $p$ and $\xi$, while $G^{1,0}_{\chi}(\xi, \phi)$ is first order in $\xi$ and zero order $\phi$; and moreover, that their support is in $\Omega^*: \mbox{supp } F_{\chi} \subset \Omega^*, \mbox{ supp } G_{\chi} \subset \Omega^*$. 
Thus, \eqref{6E2-20} becomes explicitly, still with $u = \bbm \phi \\ \xi \ebm$:
\begin{multline}        \label{6E2-21}
\int_{\Omega_{1} \cup \Omega^*} e^{2\tau \psi(x)} 
|\Delta (\chi u)(x)|^2dx \leq  C_{\lambda, e} \int_{\Omega_{1} \cup \Omega^*} e^{2\tau \psi(x)}  
[\nabla (\chi u)(x)|^2 + |(\chi u)(x)|^2]dx \\[1mm]
+ \int_{\Omega_{1} \cup \Omega^*} e^{2\tau \psi(x)} |\nabla (\chi p)(x)|^2dx +  c_\chi
\int_{\Omega^*} e^{2\tau \psi(x)} [|\nabla u(x)|^2+ |u(x)|^2 + |p(x)|^2]dx,
\end{multline}
which is the sought-after bound on the last term of the RHS of \eqref{E4-2-10}. In \eqref{6E2-21}, $c_\chi$ is a constant depending on $\chi$.

\nin \textbf{Step 9.} \textit{Final estimate for $(\chi u)$-problem \eqref{2.12}, $u = \bbm \phi \\ \xi \ebm$.} We substitute \eqref{6E2-21} into the RHS of inequality \eqref{E4-2-10}, and obtain
\begin{multline}\label{6E2-22}
\left[\rho \tau - \frac{1}{8}\right]
\int_{\Omega_{1} \cup \Omega^*} e^{2\tau \psi(x)} |\nabla (\chi u)(x)|^2dx + \left[2\rho k^2\tau^3 +\calO(\tau^2)\right] \int_{\Omega_{1} \cup \Omega^*} e^{2\tau \psi(x)} |(\chi u)(x)|^2dx
\\[1mm]
\leq 
C_{\lambda,e} \int_{\Omega_{1} \cup \Omega^*} e^{2\tau \psi(x)} [|\nabla (\chi u)(x)|^2+ |(\chi u)(x)|^2]dx
+ 3 \int_{\Omega_{1} \cup \Omega^*} e^{2\tau \psi(x)} |\nabla (\chi p)(x)|^2dx
\\[1mm]
+ \ 
c_\chi  \int_{\Omega^*} e^{2\tau \psi(x)} [|\nabla u(x)|^2+ |u(x)|^2 + |p(x)|^2]dx.
\end{multline}
Moving the first integral term on the RHS of inequality \eqref{6E2-22} to the LHS of such inequality then yields for $\tau$ sufficiently large:
\begin{multline} \label{6E2-23}
\left\{ \left[\rho \tau - \frac{1}{8}\right] - C_{\lambda,e} \right\}
\int_{\Omega_{1} \cup \Omega^*} e^{2\tau \psi(x)} |\nabla (\chi u)(x)|^2dx
\\[1mm] 
+ \left[2\rho k^2\tau^3 +\calO(\tau^2)- C_{\lambda,e} \right]
\int_{\Omega_{1} \cup \Omega^*} e^{2\tau \psi(x)} |(\chi u)(x)|^2dx
\\[1mm] 
\leq 3 \int_{\Omega_{1} \cup \Omega^*} e^{2\tau \psi(x)} |\nabla (\chi p)(x)|^2dx
+ c_{\chi}  \int_{\Omega^*} e^{2\tau \psi(x)} [|\nabla u(x)|^2+ |u(x)|^2+|p(x)|^2]dx.
\end{multline}
Inequality \eqref{6E2-23} is our final estimate for the $(\chi u)$-problem in \eqref{2.12}, \eqref{E4-2-4}, \eqref{6E2-9}.\\

\nin \textbf{Step 10.} \emph{The $(\chi p)$-problem.}
We need to estimate the first integral term on the RHS of inequality \eqref{6E2-23}. This will be accomplished in \eqref{6E2-31} below. To this end, we need to obtain preliminarily the PDE-problem satisfied by $(\chi p)$ on $G \equiv \Omega_1 \cup \Omega^*$. This task will be accomplished in this step. Accordingly, we return to the $\phi$-equation. \eqref{1.5a'}, take here the operation of ``div'' across, use $div \phi \equiv 0$ and $\div \xi \equiv 0$ from \eqref{2.2b} = \eqref{1.5c}, and obtain, recalling $\calL_1(\phi)$ in \eqref{2.6a} and $\calL_2(\xi)$ in \eqref{2.6b}
\begin{subequations}
	\begin{equation}\label{6E2-24a}
	\Delta p = -\div \calL_1(\phi) + \div \calL_2(\xi)  \mbox{ in } \Omega,
	\end{equation}
	where, actually \cite[Eqn.\,(5.21)]{RT:2008}, \cite[Eqn.\,(3.24)]{RT1:2009},
	\begin{align}
	\div \calL_1(\phi) = 2\{(\partial_x y_e \cdot \nabla) \phi \} = 2\{(\partial_x \phi \cdot \nabla)y_e\},  \label{6E2-24b}\\
	\div \calL_2(\xi) = 2\{(\partial_x B_e \cdot \nabla) \xi \} = 2\{(\partial_x \xi \cdot \nabla)B_e\}.  \label{6E2-24c}
	\end{align}
	are first-order differential operators in $\phi$ and $\xi$ respectively.
\end{subequations}
The proof of \eqref{6E2-24b} uses $\div \phi \equiv 0$ and $\div y_e \equiv 0$ in $\Omega$ from \eqref{2.2b} = \eqref{1.2c} and \eqref{1.1c}. 
Next, multiply \eqref{6E2-24a} by $\chi$. We obtain
\begin{subequations}    \label{6E2-25}
	\begin{empheq}[left=\empheqlbrace]{align}
	\Delta (\chi p) & =  -\div \calL_1(\chi \phi) + \div \calL_2(\chi \xi) + T_\chi^{0,0,1}(\phi, \xi, p) \quad \mbox{ in } \Omega; \label{6E2-25a}\\ 
	\frac{\partial(\chi p)}{\partial\nu}\bigg|_D & =  0, \quad (\chi p)|_D =0, \quad D = \partial[\Omega_1 \cup \Omega^*].    \label{6E2-25b}
	\end{empheq}
	\begin{multline}\label{6E2-25c}
	T_\chi^{0,0,1}(\phi, \xi, p)
	\equiv  [\Delta, \chi]p + [\div \calL_1, \chi] \phi - [\div \calL_2, \chi] \xi  =  \mbox{zero order in $\phi$ and $\xi$}\\
	\text{by \eqref{6E2-24b} and by \eqref{6E2-24c}; first order in $p$; supp $T_\chi^{0,0,1} \subset \Omega^*$},
	\end{multline}
\end{subequations}
while the B.C.s \eqref{6E2-25b} on the boundary $D$ defined by \eqref{E4-2-5} follow for two reasons:
\begin{enumerate}[(i)]
	\item the RHS of \eqref{E4-2-6} on $(\chi p)$ on $\partial \omega$; actually, the RHS of \eqref{E4-2-1} since $\chi \equiv 1$ on $\omega$;
	\item $\chi \equiv 0$ up to the external boundary of $\Omega^*$ and a small layer of $\Omega^*$ bordering $\Omega_0$, so that $(\chi p) = 0$, $\ds \frac{\partial (\chi p)}{\partial\nu} = 0$, on such external boundary of $\Omega^*$. 
\end{enumerate}
Thus, \eqref{6E2-25b} is justified. In \eqref{6E2-25b}, the reason for supp $T_\chi^{0,0,1} \subset \Omega^*$ is the same as in \eqref{6E2-25c} and \eqref{6E2-9b}.\\

\nin Next, we apply the pointwise Carleman estimate \eqref{E4-2-7} to problem \eqref{6E2-25a}--\eqref{6E2-25b}, that is for $w = (\chi p)$. We obtain with $G = \Omega_1 \cup \Omega^*$:
\begin{multline}
\delta_0 \left[2\rho \tau - \frac{\epsilon}{2}\right]
e^{2\tau \psi(x)} |\nabla (\chi p)(x)|^2
+ \left[4\rho k^2\tau^3(1-\delta_0) +\calO(\tau^2)\right]
e^{2\tau \psi(x)} |(\chi p)(x)|^2
\\[1mm] 
\leq  \left(1+\frac{1}{\epsilon}\right) e^{2\tau \psi(x)} |\Delta (\chi p)(x)|^2
+ \div V_{(\chi p)}(x), \quad x \in G.       \label{6E2-26}
\end{multline}
Again, it is not critical to recall what $\div V_{(\chi p)}(x)$ is, only the vanishing relationship \eqref{E4-2-8} (for $w = (\chi p)$) on an appropriate bounded domain $G$. Indeed, we shall take again $G = \Omega_1 \cup \Omega^*$, and integrate inequality \eqref{6E2-26} over $G$ (after selecting again $\delta_0 = \frac{1}{2}$, $\epsilon = \frac{1}{2}$), and obtain
\begin{multline}\label{6E2-27}
\left[\rho \tau - \frac{1}{8}\right] \int_{\Omega_{1} \cup \Omega^*} e^{2\tau \psi(x)} |\nabla (\chi p)(x)|^2dx 
+ \ \left[2\rho k^2\tau^3 +\calO(\tau^2)\right] \int_{\Omega_{1} \cup \Omega^*} e^{2\tau \psi(x)} |(\chi p)(x)|^2dx
\\[1mm] 
\leq 3 \int_{\Omega_{1} \cup \Omega^*} e^{2\tau \psi(x)} |\Delta (\chi p)(x)|^2dx
+ \int_{\partial[\Omega_{1}\cup \Omega^*]} \cancel{V_{(\chi p)} (x)} \cdot \nu \, dD,
\end{multline}
where, on the RHS of \eqref{6E2-27}, the boundary integral over $D \equiv \partial[\Omega_1 \cup \Omega^*] = [\partial \omega \, \cup$ external boundary of $\Omega^*]$, see \eqref{E4-2-5} and Fig. \ref{fig:1}, again vanishes in view of \eqref{6E2-25b} for $w = (\chi p)$. Thus, the vanishing of the last integral term of \eqref{6E2-27} is justified.\\

\nin \textbf{Step 11.} Here we now estimate the last integral term on the RHS of \eqref{6E2-27}. We multiply Eqn.\eqref{6E2-25a} by $e^{\tau \psi (x)}$, thus obtaining
\begin{align}          
e^{\tau\psi(x)} \Delta(\chi p) 
&=  -e^{\tau\psi(x)} \div \calL_1(\chi \phi) + e^{\tau\psi(x)} \div \calL_2(\chi \xi) + e^{\tau\psi(x)} T_\chi^{0,0,1}(\phi, \xi, p), \label{6E2-28} \\
e^{2\tau\psi(x)} |\Delta(\chi p)(x)|^2 &\leq c e^{2\tau\psi(x)} \left\{ |\div \calL_1(\chi \phi)(x)|^2
+ |\div \calL_2(\chi \xi)(x)|^2 \right. \nonumber \\
& \hspace{6cm} \left. + |T_\chi^{0,0,1}(\phi, \xi, p)(x)|^2 \right\}, \quad x \in G. \label{6E2-29}
\end{align}
We now integrate \eqref{6E2-29} over $G \equiv [\Omega_1 \cup \Omega^*]$. In doing so, we recall from \eqref{6E2-24b}, \eqref{6E2-24c} that $[\div \calL_1]$ and $[\div \calL_2]$ are first-order operators, and accordingly, from \eqref{6E2-25c}, that $T_\chi^{0,0,1}(\phi, \xi, p)$ is an operator which is zero order in $\phi$ and $\xi$; and first order in $p$; and that $T_\chi^{0,0,1}(\phi, \xi, p)$ has support in $\Omega^*$. We thus obtain from \eqref{6E2-29}
\begin{multline}          \label{6E2-30}
\int_{\Omega_{1} \cup \Omega^*}
e^{2\tau \psi(x)} |\Delta(\chi p)(x)|^2dx \\
\leq C_{y_{e},B_e} \int_{\Omega_{1} \cup \Omega^*}
e^{2\tau \psi(x)}\left[ \abs{\nabla(\chi \phi)(x)}^2 + \abs{(\chi \phi)(x)}^2 + \abs{\nabla(\chi \xi)(x)}^2 + \abs{(\chi \xi)(x)}^2 \right]dx\\[1mm]
+ C_\chi \int_{\Omega^*} e^{2\tau \psi(x)} \left[|\nabla p(x)|^2 + |p(x)|^2+|\phi(x)|^2 + |\xi(x)|^2 \right]dx
\end{multline}
with constant $C_\chi$ depending on $\chi$.\\

\nin \textbf{Step 12.} \textit{Final estimate of the $(\chi p)$-problem.} We now substitute \eqref{6E2-30} into the RHS of \eqref{6E2-27}, divide across by $[\rho \tau - \frac{1}{8}] > 0$ for $\tau$ large and obtain
\begin{multline}\label{6E2-31}
\int_{\Omega_{1} \cup \Omega^*}
e^{2\tau \psi(x)} |\nabla(\chi p)(x)|^2dx
+ \frac{[2\rho k^2\tau^3 + \calO(\tau^2)]}{\left[\rho\tau - \, \frac{1}{8}\right]}
\int_{\Omega_{1} \cup \Omega^*}
e^{2\tau \psi(x)} |(\chi p)(x)|^2dx
\\[1.5mm]
\leq \frac{C_{y_{e},B_e}}{\left(\rho\tau - \, \frac{1}{8}\right)}
\int_{\Omega_{1} \cup \Omega^*}
e^{2\tau \psi(x)} \left[ \abs{\nabla(\chi \phi)(x)}^2 + \abs{(\chi \phi)(x)}^2 + \abs{\nabla(\chi \xi)(x)}^2 + \abs{(\chi \xi)(x)}^2 \right]dx
\\[1mm]
+ \frac{C_\chi}{\left(\rho\tau - \, \frac{1}{8}\right)}
\int_{\Omega^*}
e^{2\tau \psi(x)} \left[|\nabla p(x)|^2 + |p(x)|^2 + |\phi(x)|^2 + |\xi(x)|^2 \right]dx.
\end{multline}
Inequality \eqref{6E2-31} is our final estimate on the $(\chi p)$-problem \eqref{6E2-25a}.\\

\nin \textbf{Step 13.} \textit{Combining the ($\chi u$)-estimate \eqref{6E2-23} with the ($\chi p$)-estimate \eqref{6E2-31}.} We return to estimate \eqref{6E2-23} and add to each side the term
$$
\frac{[2\rho k^2\tau^3+\calO(\tau^2)]}{\left[\rho\tau - \, \frac{1}{8}\right]}
\int_{\Omega_{1} \cup \Omega^*}
e^{2\tau \psi(x)} |(\chi p)(x)|^2dx
$$
to get
\begin{multline}
\left\{\left[\rho\tau - \, \frac{1}{8}\right] - C_{\lambda, e}\right\}
\int_{\Omega_{1} \cup \Omega^*}
e^{2\tau \psi(x)} |\nabla (\chi u)(x)|^2dx \\[1mm]
 + \left[2\rho k^2\tau^3+ \calO(\tau^2) - C_{\lambda, e} \right]
\int_{\Omega_{1} \cup \Omega^*}
e^{2\tau \psi(x)} |(\chi u)(x)|^2dx + \frac{\left[2\rho k^2\tau^3+ \calO(\tau^2)\right]}{\left[\rho\tau - \frac{1}{8}\right]}
\int_{\Omega_{1} \cup \Omega^*}
e^{2\tau \psi(x)} |(\chi p)(x)|^2dx
 \\[1mm]
 \leq  3 \int_{\Omega_{1} \cup \Omega^*}
e^{2\tau \psi(x)} |\nabla (\chi p)(x)|^2dx + \frac{\left[2\rho k^2\tau^3+ \calO(\tau^2)\right]}
{\left[\rho\tau - \frac{1}{8}\right]}
\int_{\Omega_{1} \cup \Omega^*}
e^{2\tau \psi(x)} |(\chi p)(x)|^2dx
 \\[1mm]
 + c_{\chi} \int_{\Omega^*}
e^{2\tau \psi(x)} [|\nabla u(x)|^2 + |u(x)|^2 + |p(x)|^2]dx. \label{6E2-32}
\end{multline}
Next, we substitute inequality \eqref{6E2-31} for the first two integral terms on the RHS of \eqref{6E2-32}, and obtain
\begin{align}
& \left\{\left[\rho\tau - \, \frac{1}{8}\right] - C_{\lambda,e} \right\}
\int_{\Omega_{1} \cup \Omega^*}
e^{2\tau \psi(x)} |\nabla (\chi u)(x)|^2dx
\nonumber \\[1mm]
& + \ 
\left\{\left[2\rho k^2\tau^3+ \calO(\tau^2)\right]- C_{\lambda,e} \right\}
\int_{\Omega_{1} \cup \Omega^*}
e^{2\tau \psi(x)} |(\chi u)(x)|^2dx
\nonumber \\[1mm]
& + \ 
\frac{\left[2\rho k^2\tau^3+ \calO(\tau^2)\right]}{\left[\rho\tau - \frac{1}{8}\right]}
\int_{\Omega_{1} \cup \Omega^*}
e^{2\tau \psi(x)} |(\chi p)(x)|^2dx
\nonumber \\[1mm]
\leq & \ 
\frac{C_{y_{e},B_e}}{\left(\rho\tau - \frac{1}{8}\right)}
\int_{\Omega_{1} \cup \Omega^*}
e^{2\tau \psi(x)} 
\left[\abs{\nabla(\chi \phi)(x)}^2 + \abs{(\chi \phi)(x)}^2 + \abs{\nabla(\chi \xi)(x)}^2 + \abs{(\chi \xi)(x)}^2 \right]dx
\nonumber \\[1mm]
& + \ 
\frac{C_\chi}{\left(\rho\tau - \frac{1}{8}\right)}
\int_{\Omega^*}e^{2\tau \psi(x)} 
\left[|\nabla p(x)|^2 + |p(x)|^2 + |\phi(x)|^2 + |\xi(x)|^2 \right]dx
\nonumber \\[1mm]
& + \ c_{\chi} \int_{\Omega^*}
e^{2\tau \psi(x)} \left[ |\nabla u(x)|^2 + |u(x)|^2 + |p(x)|^2  \right] dx. \label{6E2-33}
\end{align}
Recalling that $u= \bbm \phi \\ \xi \ebm$, we re-write \eqref{6E2-33} explicitly as
\begin{align}
& \left\{\left[\rho\tau - \, \frac{1}{8}\right] - C_{\lambda,e} \right\}
\int_{\Omega_{1} \cup \Omega^*}
e^{2\tau \psi(x)} \left[ |\nabla (\chi \phi)(x)|^2 + |\nabla (\chi \xi)(x)|^2 \right]dx \nonumber \\[3mm]
& + \ 
\left\{\left[2\rho k^2\tau^3+ \calO(\tau^2)\right]- C_{\lambda,e} \right\}
\int_{\Omega_{1} \cup \Omega^*}
e^{2\tau \psi(x)} \left[|(\chi \phi)(x)|^2 + (\chi \xi)(x)|^2 \right] dx
\nonumber \\[1mm]
& + \ 
\frac{\left[2\rho k^2\tau^3+ \calO(\tau^2)\right]}{\left[\rho\tau - \frac{1}{8}\right]}
\int_{\Omega_{1} \cup \Omega^*}
e^{2\tau \psi(x)} |(\chi p)(x)|^2dx
\nonumber \\[1mm]
\leq & \ 
\frac{C_{y_{e},B_e}}{\left(\rho\tau - \frac{1}{8}\right)}
\int_{\Omega_{1} \cup \Omega^*}
e^{2\tau \psi(x)} \left[|\nabla (\chi \phi)(x)|^2+|(\chi \phi)(x)|^2 + |\nabla (\chi \xi)(x)|^2+|(\chi \xi)(x)|^2 \right]dx
\nonumber \\[1mm]
& + \ 
\frac{C_\chi}{\left(\rho\tau - \frac{1}{8}\right)}
\int_{\Omega^*}e^{2\tau \psi(x)} 
\left[|\nabla p(x)|^2 + |p(x)|^2 + |\phi(x)|^2 + |\xi(x)|^2\right]dx
\nonumber \\[3mm]
& \hspace{-8pt} + \ c_{\chi} \int_{\Omega^*}
e^{2\tau \psi(x)} \left[ |\nabla \phi(x)|^2 + |\nabla \xi(x)|^2 + |\phi(x)|^2 + |\xi(x)|^2 +|p(x)|^2 \right] dx. \label{6E2-34}
\end{align}

\smallskip  \noindent \textbf{Step 14.} \textit{Final estimate of problem \eqref{2.2a}--\eqref{2.2b}.} Finally, we combine the integral terms with the same integrand on the LHS of \eqref{6E2-34} and obtain the final sought-after estimate which we formalize as a lemma.

\begin{lemma}
	The following inequality holds true for all $\tau$ sufficiently large:
	\begin{align}
	& \left\{\left[\rho\tau - \, \frac{1}{8}\right] - C_{\lambda, e} -
	\frac{C_{y_{e},B_e}}{\left(\rho\tau - \frac{1}{8}\right)}\right\}
	\int_{\Omega_{1} \cup \Omega^*}
	e^{2\tau \psi(x)} \left[ |\nabla (\chi \phi)(x)|^2 + |\nabla (\chi \xi)(x)|^2 \right] dx \nonumber \\[1mm]
	& + \ 
	\Bigg\{\left[2\rho k^2\tau^3+ \calO(\tau^2)\right] - C_{\lambda, e} -
	\frac{C_{y_{e},B_e}}{\left(\rho\tau- \frac{1}{8}\right)}\Bigg\}
	\int_{\Omega_{1} \cup \Omega^*}
	e^{2\tau \psi(x)} \left[|(\chi \phi)(x)|^2 + (\chi \xi)(x)|^2 \right] dx
	\nonumber \\[1mm]
	& + \ 
	\frac{\left[2\rho k^2\tau^3+ \calO(\tau^2)\right]}{\left[\rho\tau - \frac{1}{8}\right]}
	\int_{\Omega_{1} \cup \Omega^*}
	e^{2\tau \psi(x)} |(\chi p)(x)|^2dx
	\nonumber \\[1mm]
	\leq & \ 
	\frac{C_\chi}{\left(\rho\tau - \frac{1}{8}\right)}
	\int_{\Omega^*}
	e^{2\tau \psi(x)} 
	\left[|\nabla p(x)|^2 + |p(x)|^2 + |\phi(x)|^2 + |\xi(x)|^2\right]dx
	\nonumber \\[1mm]
	& + \ c_{\chi} \int_{\Omega^*}
	e^{2\tau \psi(x)} \left[ |\nabla \phi(x)|^2 + |\nabla \xi(x)|^2 + |\phi(x)|^2 + |\xi(x)|^2 + |p(x)|^2 \right] dx. \label{6E2-35}
	\end{align}
\end{lemma}
\noindent
We note explicitly two critical features of estimate \eqref{6E2-35}: the integral terms on its LHS are over $[\Omega_1 \cup \Omega^*]$; while the integral terms on its RHS are over $\Omega^*$. As already noted, \eqref{6E2-35} is the ultimate estimate regarding the original problem \eqref{2.2a}--\eqref{2.2b}.\\ 

\noindent \textbf{Step 15.} \textit{The choice of weight function $\psi(x)$.}  
We now choose the strictly convex function $\psi(x)$ as follows (Fig. \ref{fig:3} and Fig. \ref{fig:4}, as well as Fig. \ref{fig:2}):
\begin{align}
\psi(x) & \geq  0 \mbox{ on $\Omega_1$ where $\chi \equiv 1$ by \eqref{E4-2-2}, so that $e^{2\tau \psi(x)}\geq 1$ on $\Omega_1$},  \label{6E2-36}
\\[3mm]
\psi(x) & \leq  0 \mbox{ on $\Omega_0 \cup \Omega^*$; where $\chi < 1$, so that $e^{2\tau \psi(x)}\leq 1$ on $\Omega^*$}, \label{6E2-37}
\end{align}
in such a way that $\psi(x)$ has no critical point in $\Omega \setminus \omega$, as required by Theorem \ref{thm4-2-1} ($\psi$ has no critical points on $G=\Omega_1 \cup \Omega^*$): that is, the critical point(s) of $\psi$ will fall on $\omega$, outside the region $G = \Omega_1 \cup \Omega^*$ where we have integrated.

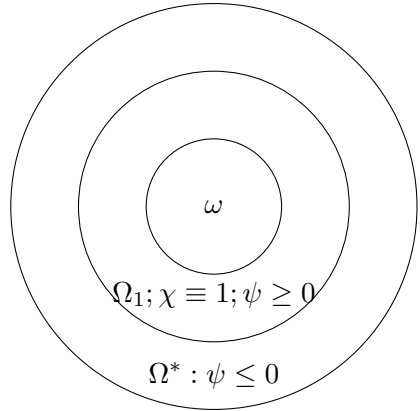
\begin{figure}[H]
	\centering
	\begin{tikzpicture}[scale=0.45]
	\draw (0,0) node {$\omega$} circle [radius = 2];
	\draw (0,0) circle [radius = 4];
	\draw (0,0) circle [radius = 6];
	
	\draw (0,-2.6) node {$\Omega_1; \chi \equiv 1; \psi \geq 0$};
	\draw (0,-5) node {$\Omega^*: \psi \leq 0$};
	\end{tikzpicture}
	\caption{Case 1: construction of $\Omega_1$ and $\Omega^*$}
	\label{fig:3}
\end{figure}

\begin{figure}[H]
	\centering
	\begin{picture}(220,120)
	\put(0,50){\line(1,0){175}}
	\multiput(0,48)(60,0){3}{\line(0,1){5}}
	\put(120,50){\line(0,1){50}}
	\put(120,100){\line(1,0){120}}
	
	\bezier{1000}(0,50.5)(40,50.5)(60,50.5)
	\bezier{1000}(60,50.5)(65,50.5)(70,50.5)
	\bezier{1000}(70,50.5)(94,50.5)(94,80)
	\bezier{1000}(94,80)(95,100)(120,100)
	
	\put(130,20){\vector(-1,0){130}}
	\put(135,17){$\Omega$}
	\put(147,20){\vector(1,0){150}}
	
	\put(193,90){$\chi \equiv 1$}
	
	\put(20,35){\vector(-1,0){19}}
	\put(24,33){$\Omega_0$}
	\put(40,35){\vector(1,0){19}}
	
	\put(81,35){\vector(-1,0){19}}
	\put(84,33){$\Omega^*$}
	\put(99,35){\vector(1,0){18}}
	
	\put(135,35){\vector(-1,0){15}}
	\put(139,33){$\Omega_1$}
	\put(154,35){\vector(1,0){19}}
	
	\put(200,50){\vector(-1,0){22}}
	\put(203,47){$\omega$}
	\put(215,50){\vector(1,0){22}}
	
	{
		\bezier{1000}(-40,-30)(50,-50)(170,109)
		\bezier{1000}(170,109)(212,146)(245,109)
		\bezier{1000}(245,109)(260,80)(260,60)
		\bezier{1000}(260,60)(263,-20)(290,-20)
	}
	
	\put(35,-30){$\psi$}
	\end{picture}
	\vspace*{30pt}
	\caption{Case 1: choice of $\psi$}
	\label{fig:4}
\end{figure}
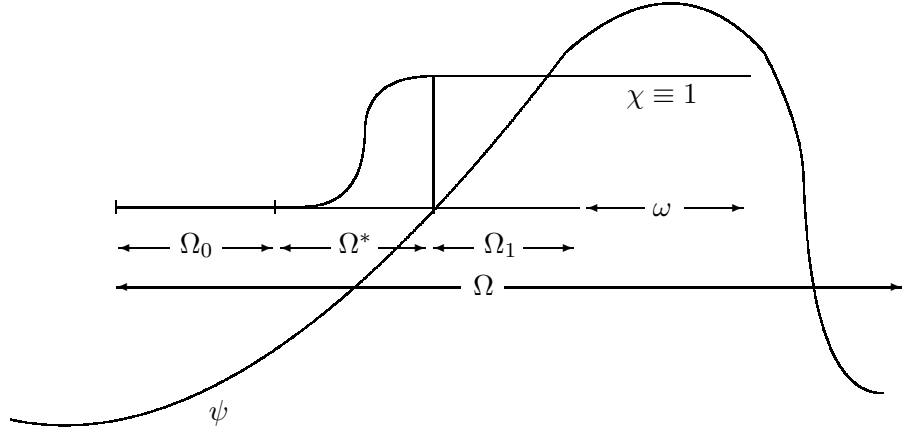

Having chosen $\psi(x)$ as in \eqref{6E2-36}, \eqref{6E2-37} with no critical points in $\Omega \setminus \omega$--i.e., no critical points on $G = \Omega_1 \cup \Omega^*$--we return to the basic estimate \eqref{6E2-35}, with $\tau$ sufficiently large (Fig. \ref{fig:4}). On the LHS of \eqref{6E2-35}, we retain only integration over $\Omega_1$, where $\psi \geq 0$, hence $e^{2\tau \psi} \geq 1$ and $\chi \equiv 1$ by \eqref{E4-2-2}, so that $(\chi u) \equiv u$ on $\Omega_1$, that is, $(\chi \phi) \equiv \phi$ on $\Omega_1$ and $(\chi h) \equiv h$ on $\Omega_1$. On the RHS of \eqref{6E2-35} we have $\psi \leq 0$ on $\Omega^*$, hence $e^{2\tau \psi} \leq 1$ on $\Omega^*$. We thus obtain from \eqref{6E2-35} for $\tau$ sufficiently large
\begin{align}
& \left\{\left[\rho\tau - \, \frac{1}{8}\right] - C_{\lambda, e} -
\frac{C_{B_e,y_{e}}}{\left(\rho\tau- \frac{1}{8}\right)}\right\}
\int_{\Omega_{1}}
\left[|\nabla \phi(x)|^2 + |\nabla \xi(x)|^2\right] dx
\nonumber \\[3mm]
& + \ 
\left\{\left[2\rho k^2\tau^3+ \calO(\tau^2)\right]- C_{\lambda, e}
- \frac{6 C_{B_e,y_{e}}}{\left(\rho\tau- \frac{1}{8}\right)}\right\}
\int_{\Omega_{1}} \left[|\phi(x)|^2 + |\xi(x)|^2 \right]dx     \nonumber \\
& 
+ \ 
\frac{\left[2\rho k^2\tau^3+ \calO(\tau^2)\right]}{\left[\rho\tau - \frac{1}{8}\right]} \int_{\Omega_{1}} |p(x)|^2dx
\nonumber \\[3mm]
\leq & \  \frac{C_\chi}{\left(\rho \tau - \frac{1}{8}\right)}
\int_{\Omega^*}\left[|\nabla p(x)|^2+|p(x)|^2+|\phi(x)|^2 + |\xi(x)|^2\right]dx
\nonumber \\[3mm]
& + \ c_{\chi} \int_{\Omega^*}
\left[ |\nabla \phi(x)|^2 + |\nabla \xi(x)|^2 + |\phi(x)|^2 + |\xi(x)|^2+|p(x)|^2 \right] dx.      \label{6E2-38}
\end{align}
For $\tau$ sufficiently large, inequality \eqref{6E2-38} is of the type
\begin{subequations}
	\label{6E2-39}
	\begin{align}
	&\left(\tau - \mbox{const} - \frac{1}{\tau}\right)
	\int_{\Omega_{1}} \left[|\nabla \phi(x)|^2 + |\nabla \xi(x)|^2\right] dx \nonumber \\
	& + 
	\left(\tau^3 - \mbox{const} - \frac{1}{\tau}\right)
	\int_{\Omega_{1}} \left[|\phi(x)|^2 + |\xi(x)|^2\right] dx + (\tau^2) \int_{\Omega_{1}} |p(x)|^2dx
	\nonumber \\[3mm]
	\leq & \ 
	\frac{c}{\tau} \int_{\Omega^*} \left[|\nabla p(x)|^2+|p(x)|^2+|\phi(x)|^2 + |\xi(x)|^2\right]dx
	\nonumber \\[3mm]
	& + \ \mbox{const} \int_{\Omega^*} \left[ |\nabla \phi(x)|^2 + |\nabla \xi(x)|^2 + |\phi(x)|^2 + |\xi(x)|^2+|p(x)|^2 \right] dx \label{6E2-39a}
	\end{align}
	or setting as usual $u=\{\phi, h\}$, we re-write \eqref{6E2-39a} equivalently as
	\begin{align}
	& \left(\tau - \mbox{const} - \frac{1}{\tau}\right)
	\int_{\Omega_{1}} |\nabla u(x)|^2  dx + 	\left(\tau^3 - \mbox{const} - \frac{1}{\tau}\right)
	\int_{\Omega_{1}} |u(x)|^2 dx + (\tau^2) \int_{\Omega_{1}} |p(x)|^2dx
	\nonumber \\[3mm]
	\leq & \ \frac{c}{\tau} \int_{\Omega^*} \left[|\nabla p(x)|^2+|p(x)|^2 + |u(x)|^2\right]dx
	 + \ \mbox{const} \int_{\Omega^*} \left[|\nabla u(x)|^2 + |u(x)|^2 + |p(x)|^2 \right] dx          \label{6E2-39b}
	\\[3mm]
	\leq & \ \frac{c}{\tau} C_1(p,u;\Omega^*) + \mbox{const } C_2(p,u;\Omega^*).          \label{6E2-39c}
	\end{align}
\end{subequations}
In going from \eqref{6E2-39b} to \eqref{6E2-39c}, we have emphasized in the notation that we are working with a fixed solution $\{u, p\}$ of problem \eqref{2.2a}--\eqref{2.2b}, so that the integrals on the RHS of \eqref{6E2-39b} are fixed numbers $C_1(p,u;\Omega^*)$ and $C_2(p,u;\Omega^*)$, depending on such fixed solution $\{u,p\}$ as well as $\Omega^*$, $u = \{\phi, \xi\}$.  Inequality \eqref{6E2-39} is more than we need. On its LHS, we may drop the $\nabla u$-term over $\Omega_1$; and alternatively either keep only the $u$-term over $\Omega_1$, and divide the remaining inequality across by $(\tau^3-\mbox{const}-\frac{1}{\tau})$ for $\tau$ large; or else keep only the $p$-term over $\Omega_1$ and divide the corresponding inequality across by $\tau^2$. We obtain, respectively,
\begin{eqnarray}
\int_{\Omega_{1}} |u(x)|^2dx
& \leq & \left( \frac{C}{\tau^3} \ \frac{1}{\tau}\right) C_1(p,u;\Omega^*) + \frac{\rm const}{\tau^3} C_2(p,u;\Omega^*) \rightarrow 0,              \label{6E2-40}
\\[3mm]
\int_{\Omega_{1}} |p(x)|^2dx
& \leq & \left( \frac{C}{\tau^2} \ \frac{1}{\tau}\right) C_1(p,u;\Omega^*) + \frac{\rm const}{\tau^2} C_2(p,u;\Omega^*) \rightarrow 0         \label{6E2-41}
\end{eqnarray}
as $\tau \rightarrow + \infty$. We thus obtain
\begin{equation}\label{6E2-42}
u(x) \equiv \{\phi(x), \xi(x)\} \equiv 0 \mbox{ in } \Omega_1; \qquad p(x) \equiv 0  \mbox{ in } \Omega_1.
\end{equation}
and recalling \eqref{2.2b} and Step 1
\begin{equation}    \label{NE2-43}
u(x) \equiv \{\phi(x), \xi(x)\} \equiv 0, \quad
p(x) \equiv 0 \mbox{ in } \omega \cup \Omega_1.
\end{equation}
The implication: Step 1 $\Longrightarrow$ \eqref{NE2-43} is illustrated by Fig. \ref{fig:5}, with $u = \bbm \phi \\ \xi \ebm$.

\begin{figure}[H]
	\centering
	\begin{tikzpicture}[scale=0.4]
	\draw (0,0) circle (2);
	\draw (0,0) circle (6);
	\draw (0,0.5) node {$u \equiv 0$};
	\draw (0,-0.5) node {$p \equiv 0$};
	\draw (2.5,2) node {$\omega$};
	\draw[->] (2.2,1.7) -- (1.2,0.7);
	\draw (-2.5,-2) node {$\Omega_1$};
	\draw[->] (0,-1) .. controls (1,-2) and (0.5,-2.5) .. (-0.3,-3) node[anchor = north] {$u \equiv 0; p \equiv 0$};
	\end{tikzpicture}
	\caption{Case 1: from $\{u,p\} = 0$ on $\omega$ to $\{u,p\} = 0$ on $\Omega_1$}
	\label{fig:5}
\end{figure}
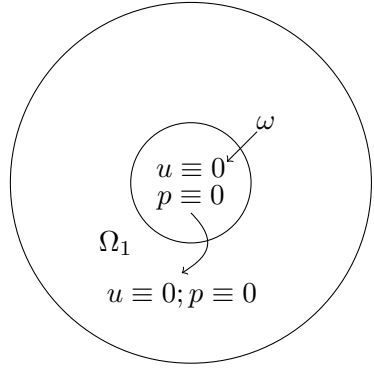

Finally, we can now push the external boundary of $\Omega_1$ as close as we please to the boundary $\partial\Omega$ of $\Omega$, and thus we finally obtain
\begin{equation}
u(x) \equiv \{\phi(x), \xi(x)\} \equiv 0 \mbox{ in } \Omega, \qquad 
p(x) \equiv 0 \mbox{ in } \Omega.
\end{equation}
Indeed, we have $u \equiv \{\phi, \xi \} \in (W^{2,q}(\Omega) \cap W_{0}^{1,q}(\Omega))^d \times (W^{2,q}(\Omega))^d$ and $p \in  W^{1,q}(\Omega)$. Moreover, $W^{1,q}(\Omega) \hookrightarrow C(\overline{\Omega})$ for $q>d$ \cite[p. 78]{SK:1989} as assumed, and more generally $W^{m,q}(\Omega) \hookrightarrow C^k(\overline{\Omega})$ for $qm > d$, $\ds k = m-\frac{d}{q}$ \cite[p. 79]{SK:1989}. A fortiori, $u \in (C(\overline{\Omega}))^d, p \in C(\overline{\Omega})$, $q > d$, as assumed. Thus, if it should happen that $u(x_1) \neq 0$ at a point $x_1 \in \Omega$ near $\partial\Omega$, hence $u(x) \not\equiv 0$ in a suitable neighborhood $N$ of $x_1$, then it would suffice to take $\Omega_1$ as to intersect such $N$ to obtain a contradiction. Theorem \ref{Thm-UCP-1} is proved at least in the Case 1 (Fig. \ref{fig:1}, Fig. \ref{fig:2}, and Fig. \ref{fig:localized_pair}).\\

\nin {\bf Case 2}: Let $\omega$ be a full collar of boundary $\Gamma = \partial\Omega$ (Fig.\ref{fig:6} and Fig.\ref{fig:7}). Then, the above proof of Case 1 can be carried out with sets $\Omega_1$, $\Omega^*$, and $\Omega_0$, as indicated in Fig.\ref{fig:6}.

Let now $\omega$ be a partial collar of the boundary $\Gamma = \partial \Omega$. Then, the above proof of Case 1 can be carried out with sets $\Omega_1$, $\Omega^*$, and $\Omega_0$, as indicated in Fig.\ref{fig:8}.
\hfill $\qedsymbol$

\begin{figure}[H]
	\centering
	\begin{tikzpicture}[scale=0.6]
	\draw (0,0) node {$\Omega_0; \chi \equiv 0$} circle [x radius=2, y radius=1];
	\draw (0,0) circle [x radius=4, y radius=2];
	\draw (0,0) circle [x radius=6, y radius=3];
	\draw (0,0) circle [x radius=8, y radius=4];
	
	\draw \foreach \x in {0,1,2,3}
	{
		(-7.5+0.09*\x, 1.4+0.1*\x) -- (-7.2+0.1*\x, 1.1+0.09*\x) 
	};
	
	\draw \foreach \x in {0,1,2,3}
	{
		(-7.5+0.09*\x, -1.4-0.1*\x) -- (-7.2+0.1*\x, -1.1-0.09*\x) 
	};
	
	\draw[->] (2, 0) -- (1.3, 0) node[anchor = north west] {$\nu$};
	\draw[->] (6, 0) -- (6.7, 0) node[anchor = north east] {$\nu$};
	\draw[thick, ->] (1,-5) node[anchor = west] {$D$} -- (1,-0.85);
	\draw[thick, ->] (1,-5) -- (-3.5,-2.45);
	\draw (-5,0.5) node {$\Omega_1$};
	\draw (-5,-0.5) node {$\chi \equiv 1$};
	\draw (-3,0) node {$\Omega^*$};
	\draw (-7,0.5) node {$\omega$};
	\draw (-7,-0.5) node {$\chi \equiv 1$};
	\end{tikzpicture}
	\caption{Case 2: $\omega$ is a collar of $\Gamma$; $G = \Omega_1 \cup \Omega^*; \ \partial G = D = [\mbox{internal boundary of } \omega] \cup [\mbox{internal boundary of } \Omega^*]$}
	\label{fig:6}
\end{figure}
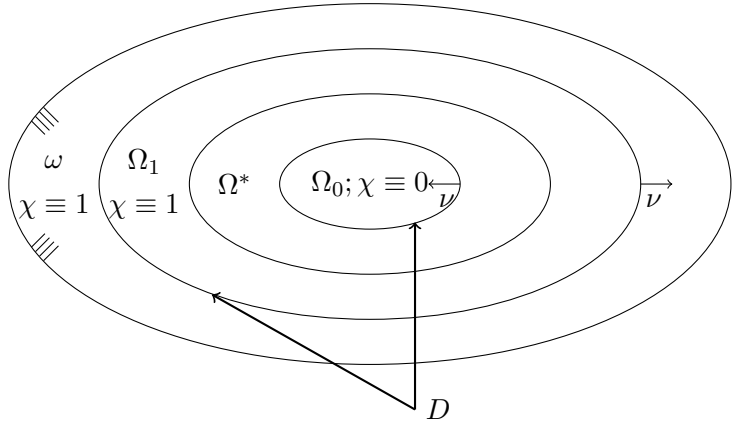
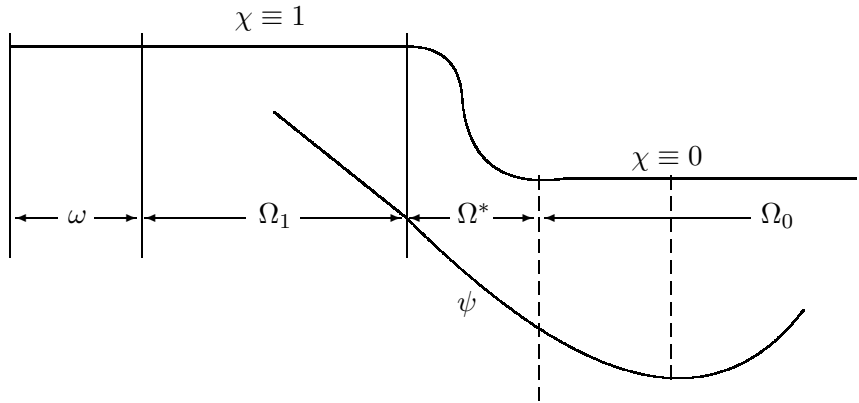
\begin{figure}[H]
	\centering
	\begin{picture}(320,170)
	\put(260,50){\line(1,0){60}}
	\put(0,20){\line(0,1){85}}
	\put(50,20){\line(0,1){85}}
	\put(150,20){\line(0,1){85}}
	\put(0,100){\line(1,0){150}}
	\multiput(200,51)(0,-8){11}{\line(0,-1){5}}
	\multiput(250,51)(0,-8){10}{\line(0,-1){5}}
	
	\bezier{1000}(100,75)(125,55)(150,35)
	\bezier{1000}(150,35)(250,-65)(300,0)
	
	\bezier{1000}(150,100)(170,100)(171,80)
	\bezier{1000}(171,80)(174,45)(210,50)
	\bezier{1000}(210,50)(210,50)(270,50)
	
	\put(85,108){$\chi \equiv 1$}
	\put(235,55){$\chi \equiv 0$}
	
	\put(18,35){\vector(-1,0){16}}
	\put(22,33){$\omega$}
	\put(32,35){\vector(1,0){16}}
	
	\put(89,35){\vector(-1,0){38}}
	\put(94,33){$\Omega_1$}
	\put(110,35){\vector(1,0){38}}
	
	\put(166,35){\vector(-1,0){15}}
	\put(169,33){$\Omega^*$}
	\put(183,35){\vector(1,0){15}}
	
	\put(280,35){\vector(-1,0){78}}
	\put(284,33){$\Omega_0$}
	
	\put(169,0){$\psi$}
	
	\end{picture}
	\vspace*{30pt}
	\caption{Case 2: the choice of $\psi$ for Fig. \ref{fig:6}}
	\label{fig:7}
\end{figure}

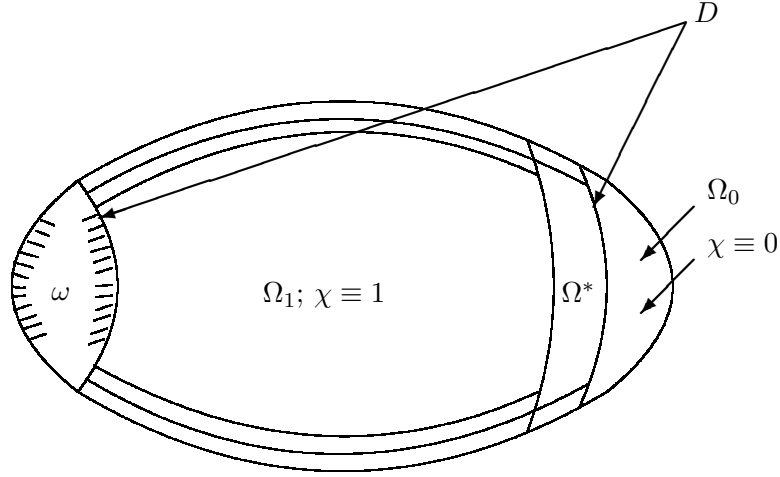
\begin{figure}[H]
	\centering
	\begin{picture}(290,170)
	\bezier{1000}(50,10)(150,-50)(250,10)
	\bezier{1000}(53,15)(146,-41)(243,13)
	\bezier{1000}(56,20)(142,-28)(224,10)
	\bezier{1000}(50,90)(150,150)(250,90)
	\bezier{1000}(53,85)(146,140)(243,88)
	\bezier{1000}(57,80)(142,130)(224,92)
	
	\bezier{1000}(50,10)(0,50)(50,90)
	\bezier{1000}(250,10)(300,50)(250,90)
	\bezier{1000}(50,90)(80,50)(50,10)
	\bezier{1000}(240,95)(260,50)(240,5)
	\bezier{1000}(220,105)(240,50)(220,-5)
	\put(40,45){$\omega$}
	\put(120,44){$\Omega_1$; $\chi\equiv 1$}
	\put(233,44){$\Omega^*$}
	\put(288,83){$\Omega_0$}
	\put(288,63){$\chi\equiv 0$}
	\put(283,80){\vector(-1,-1){20}}
	\put(283,60){\vector(-1,-1){20}}
	\put(283,150){$D$}
	\thicklines
	\put(280,150){\vector(-1,-2){35}}
	\put(280,150){\vector(-3,-1){222}}
	\put(283,80){\vector(-1,-1){20}}
	\put(283,60){\vector(-1,-1){20}}
	\thinlines
	\bezier{40}(36,75)(39,74)(41,73)
	\bezier{40}(34,71)(37,70)(39,69)
	\bezier{40}(31,67)(34,66)(37,65)
	\bezier{40}(28,63)(31,62)(35,61)
	\bezier{40}(26,59)(30,58)(33,57)
	\bezier{40}(25,55)(28,54)(31,53)
	\bezier{40}(25,50)(27,49.25)(30,49.5)
	\bezier{40}(25,46)(27,46.25)(30,46.5)
	\bezier{40}(26,41)(29,42)(32,43)
	\bezier{40}(28,37)(31,38)(34,39)
	\bezier{40}(30,34)(33,35)(36,36)
	\bezier{40}(32,30)(35,31)(38,32)
	
	\bezier{40}(52,75)(55,76)(57,77)
	\bezier{40}(54,71)(57,72)(59,73)
	\bezier{40}(54,67)(57,68)(60,69)
	\bezier{40}(55,63)(58,64)(62,65)
	\bezier{40}(56,59)(59,60)(62,61)
	\bezier{40}(57,55)(60,56)(63,57)
	\bezier{40}(57,50)(60,50.25)(63,50.5)
	\bezier{40}(57,46)(60,46.25)(63,46.5)
	\bezier{40}(57,41)(60,42)(63,43)
	\bezier{40}(56,37)(59,38)(62,39)
	\bezier{40}(55,32)(58,33)(61,34)
	\bezier{40}(54,28)(57,29)(60,30)
	\end{picture}
	\vspace*{30pt}
	\caption{Case 2: $\omega$ is a collar of a portion of the boundary}
	\label{fig:8}
\end{figure}

\section{Implication of Theorem \ref{Thm-UCP-1} on the solution of the corresponding uniform stabilization problem of the MHD by finite dimensional interior localized static feedback controllers.}\label{Sec-3}

As is by now well known \cite{RT:2024}, a result such as Theorem \ref{Thm-UCP-1} is the ``ignition key" to solve a corresponding stabilization problem. Because of space constraints, we can only report here, in a very concise form, the direct implication of Theorem \ref{Thm-UCP-1} in the establishing (in the adjoint version) the Kalman algebraic condition for the finite dimensional unstable component of the linearized dynamics (\hyperref[1.4]{1.4a-e}). Solution of the full local uniform stabilization problem of the original nonlinear problem (\hyperref[1.1]{1.1a-e}) is the Besov space setting, by means of two finite dimensional localized static controllers $\{ u,v \}$ in feedback form and of minimal dimension is given in \cite[$d = 2,3$]{LPT.6}.\\

\nin \textbf{Preliminaries} We preliminarily assume that the space $\bs{L}^p(\Omega)$ can be decomposed into the direct (non orthonormal sum for $q \neq 2$)
\begin{subequations}\label{3.1}
	\begin{align}
	\bs{L}^q(\Omega) &= \Lso \oplus \bs{G}^q(\Omega),\\
	\Lso &= \overline{\{ \bs{y} \in \bs{C}_c^{\infty}(\Omega): \div \ \bs{y} = 0 \text{ in } \Omega \}}^{\norm{\cdot}_q} \nonumber\\
	&= \{\bs{g} \in \bs{L}^q(\Omega): \div \ \bs{g} = 0; \  \bs{g} \cdot \nu = 0 \text{ on } \partial \Omega \}, \nonumber\\
	& \hspace{3cm} \text{ for any locally Lipschitz domain } \Omega \subset \mathbb{R}^d, d \geq 2 \\
	\bs{G}^q(\Omega) &= \{\bs{y} \in \bs{L}^q(\Omega): \bs{y} = \nabla p, \ p \in W_{loc}^{1,q}(\Omega) \} \ \text{where } 1 \leq q < \infty.
	\end{align}
\end{subequations}
\nin of the solenoidal vector space $\Lso$ and the space of gradient fields (Helmholtz decomposition). Both of these are closed subspaces of $\bs{L}^q$. This is a very weak assumption in the bounded domain $\Omega$ in $\BR^d, \ d = 2,3$. Let $P_q$ be the Helmholtz projection, the unique linear, bounded, idempotent ($P_q^2 = P_q$) projection operator $P_q: \bs{L}^q(\Omega)$ onto $\Lso$, having $\bs{G}^q(\Omega)$ as its null space. Introduce the operators
\begin{align}
A_{1,q} w &= -P_q \Delta w, \quad
\mathcal{D}(A_{1,q}) = \bs{W}^{2,q}(\Omega) \cap \bs{W}^{1,q}_0(\Omega) \cap \Lso, \label{3.2}\\
A_{2,q} \BW &= -\Delta \BW, \quad
\mathcal{D}(A_{2,q}) = \left\{ \BW \in \bs{W}^{2,q}(\Omega) \cap \Lso, \ (\curl \ \BW) \times n \equiv 0 \text{ on } \Gamma \right\}, \label{3.3}\\
A_{o,y_e,q} w &= P_q \calL^+_{y_e} w = P_q \left[ (y_e \cdot \nabla) w + (w \cdot \nabla)y_e \right], \quad \calD(A_{o,y_e,q}) = \calD(A_{1,q}^{\rfrac{1}{2}}) \subset \Lso, \label{3.4}\\
A_{o,B_e,q} \BW &= P_q \calL^+_{B_e}\BW = P_q \left[ (B_e \cdot \nabla) \BW + (\BW \cdot \nabla)B_e \right], \quad \calD(A_{o,B_e,q}) = \calD(A_{2,q}^{\rfrac{1}{2}}) \subset \Lso, \label{3.5}\\
L^-_{B_e}w &= \calL^-_{B_e}w = \left[ (B_e \cdot \nabla)w - (w \cdot \nabla)B_e  \right], \quad \calD(L^-_{B_e}) = \calD(A_{2,q}^{\rfrac{1}{2}}) \subset \Lso, \label{3.6}\\
L^-_{y_e}\BW &= \calL^-_{y_e} \BW = (y_e \cdot \nabla) \BW - (\BW \cdot \nabla)y_e, \quad \calD(L^-_{y_e}) = \calD(A_{2,q}^{\rfrac{1}{2}}) \subset \Lso. \label{3.7}
\end{align}
Invoking \eqref{3.5}, \eqref{3.6}, we rewrite (\hyperref[1.4]{1.4a-b}) more conveniently as
\begin{empheq}[left=\empheqlbrace]{align}
w_t - \nu \Delta w + \calL^+_{y_e}(w) - \calL^+_{B_e}(\mathbb{W}) + \nabla p &= mu &\text{ in } Q, \label{3.8}\\
\mathbb{W}_t - \eta \Delta \mathbb{W} + \calL^-_{y_e}(\mathbb{W}) + \calL^-_{B_e}(w) &= mv &\text{ in } Q, \label{3.9}
\end{empheq}
to be accompanied by the divergence free contribution \eqref{1.4c} and the B.C. \eqref{1.4d}. We next invoke the Helmholtz projection $P_q: \LpO$ onto $\Lso$ to eliminate the pressure term $\nabla p$ in \eqref{1.9a}, taking advantage of the divergence free conditions $\div \ w \equiv 0$, $\div \ \mathbb{W} \equiv 0$ in $Q$, and the conditions $w \equiv 0$ and $\mathbb{W}\cdot n \equiv 0$ on $\Sigma$, which are intrinsic conditions in the $\Lso$-space. We obtain

\begin{equation}\label{3.10}
w_t - \nu( P_q\Delta) w + (P_q\calL^+_{y_e})(w) + (P_q\calL^+_{B_e})(\mathbb{W}) = mP_qu \quad \text{ in } Q
\end{equation}
as $P_q\nabla p \equiv 0$, which along with Eq \eqref{3.9} for $\mathbb{W}$ yields the following first order PDE system\\
\begin{equation}\label{3.11}
\frac{d}{dt} \bbm w \\[1mm] \BW \ebm = \bbm \nu P_q\Delta & 0 \\[1mm] 0 & \eta \Delta \ebm \bbm w \\[1mm] \BW \ebm + \bbm -P_q \calL^+_{y_e} & P_q \calL^+_{B_e} \\[1mm] -\calL^-_{B_e} & -\calL^-_{y_e} \ebm \bbm w \\ \BW \ebm + \bbm mP_qu \\[1mm] mP_qv \ebm, 
\end{equation}
\nin along with (\hyperref[1.4]{1.4c-d}). Thus, in view of \eqref{3.2}-\eqref{3.5}, the abstract version of the PDE-coupled problem (\hyperref[1.4]{1.14a-d}) is
\begin{equation}\label{3.12}
\frac{d}{dt} \bbm w \\[1mm] \BW \ebm = \bbm -\nu A_{1,q} & 0 \\[1mm] 0 & -\eta A_{2,q} \ebm \bbm w \\[1mm] \BW \ebm + \bbm -A_{o,y_e,q} & A_{o,B_e,q} \\[1mm] L^-_{B_e} & -L^-_{y_e} \ebm \bbm w \\[1mm] \BW \ebm + \bbm mP_qu \\[1mm] mP_qv \ebm,
\end{equation}
finally 
\begin{equation}\label{3.13}
\frac{d}{dt} \bbm w \\ \BW \ebm = \wti \BA_q \bbm w \\ \BW \ebm + \bbm mP_qu \\ mP_qv \ebm, \quad \bos{\eta} = \bbm w \\ \BW \ebm, \quad \text{ on } \Yqs = \Lso \times \Lso.
\end{equation}
\begin{multline}\label{3.14}
\wti \BA_q = \BA_{o,q} + \Pi = \bbm -\nu A_{1,q} & 0 \\ 0 & -\eta A_{2,q} \ebm + \bbm -A_{o,y_e,q} & A_{o,B_e,q} \\ L^-_{B_e} & -L^-_{y_e} \ebm\\
\Lso \times \Lso \supset \calD(\wti \BA_q) = \calD(A_{1,q}) \times \calD(A_{2,q}) \to \bs{Y}^q_{\sigma}(\Omega)
\end{multline}
The operator $\wti \BA_q$ is the generator a s.c., analytic semigroup $e^{\wti \BA_qt}$ on $\Lso \times \Lso \equiv \bs{Y}^q_{\sigma}(\Omega)$, \cite{LPT.6}.

\subsection{Introduction to the stabilization problem}	

\nin For the problem of stabilization to be relevant, the assumption is that: the generator $\wti \BA_q$ of a s.c. analytic compact semigroup is \uline{unstable} on $\Lso \times \Lso \equiv \bs{Y}^q_{\sigma}(\Omega)$, in the sense that there are $N$ unstable eigenvalues $\lambda_1, \lambda_2, \dots, \lambda_N$ of $\wti \BA_q$

\begin{figure}[!ht]
	\centering
	\begin{tikzpicture}[x=10pt,y=10pt,>=stealth, scale=.6]
	\draw  (-15,0) -- (15,0);
	\draw  (0,-10) -- (0,10);
	\draw[thick]  (10,0) -- (-10,7);
	\draw[thick]  (10,0) -- (-10,-7);
	\draw (-3,1.6) node {$\bullet$};
	\draw (-6,2) node {$\lambda_{N+1}$};
	\draw (1.5,1.3) node {$\bullet$};
	\draw (4.5,0.8) node {$\bullet$};
	\draw (8,5.5) node {$\lambda_{1}$};
	\draw (4.5,6) node {$\lambda_{N}$};
	\draw[->]  (7,5) -- (4.7,1.1);
	\draw[->]  (4,5.5) -- (1.8,1.7);
	\end{tikzpicture}
	\caption{The eigenvalues of $\wti \BA_q$}
	\label{fig:9}
\end{figure}
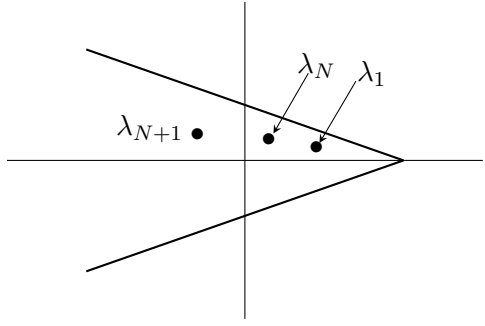
\begin{equation*}
\ldots \leq Re~ \lambda_{N+2} \leq Re~ \lambda_{N+1} < 0 \leq Re~\lambda_N \leq  \ldots \leq Re~\lambda_2 \leq Re~\lambda_1
\end{equation*}
\nin where the eigenvalues of $\wti \BA_q$ are numbered in order of decreasing real parts. Let $M$ be the number of distinct unstable eigenvalues of $\wti \BA_q$ (or $\wti \BA_q^*$). For each $i = 1, \dots, M$, we denote by  $\ds \{\boldsymbol{\Phi}_{ij}\}_{j=1}^{\ell_i} = \left\{  \bbm \varphi_{ij} \\ \psi_{ij} \ebm \right\}_{j = 1}^{\ell_i}, \{\boldsymbol{\Phi}^*_{ij}\}_{j=1}^{\ell_i} = \left\{ \bbm \varphi^*_{ij} \\ \psi^*_{ij} \ebm  \right\}_{j = 1}^{\ell_i}$ the normalized, linearly independent eigenfunctions of $\wti \BA_q$, respectively $\wti \BA_q^*$, say, on
\begin{multline}\label{3.15}
\bs{Y}^q_{\sigma}(\Omega) \equiv \Lso \times \Lso
\mbox{ and } \\
(\bs{Y}^q_{\sigma}(\Omega))^* \equiv (\Lso)' \times (\Lso)' = \bs{L}^{q'}_{\sigma}(\Omega) \times \bs{L}^{q'}_{\sigma}(\Omega), \quad
\frac{1}{q} + \frac{1}{q'} = 1,
\end{multline}
corresponding to the $M$ distinct unstable eigenvalues $\lambda_1, \ldots, \lambda_M$ of $\wti \BA_q$ and $\overline{\lambda}_1, \ldots, \overline{\lambda}_M$ of $\wti \BA_q^*$ respectively,:
\begin{align}
\wti \BA_q \boldsymbol{\Phi}_{ij} &= \lambda_i \boldsymbol{\Phi}_{ij} \in \calD(\wti \BA_q) = \calD(A_{1,q}) \times \calD(A_{2,q}) \nonumber\\
& =  [\bs{W}^{2,q}(\Omega) \cap \bs{W}^{1,q}_0(\Omega) \cap \Lso] \times [\bs{W}^{2,q}(\Omega) \cap \Lso]  \label{3.16}\\
\wti \BA_q^* \boldsymbol{\Phi}_{ij}^* &= \bar{\lambda}_i \boldsymbol{\Phi}_{ij}^* \in \calD(\wti \BA_q^*) = [\bs{W}^{2,q'}(\Omega) \cap \bs{W}^{1,q'}_0(\Omega) \cap \bs{L}^{q'}_{\sigma}(\Omega)] \times [\bs{W}^{2,q'}(\Omega) \cap \bs{L}^{q'}_{\sigma}(\Omega)]. \label{3.17}
\end{align}
\begin{equation*}      
\begin{tikzpicture}
\draw (0,0) node {$\bar{\lambda}_i$};
\draw (-0.4,-0.4) -- (-1,-1) node[anchor = north east] {$\bs{\Phi}^*_{i1}$,};
\draw (-0.2,-0.4) -- (-0.4,-1) node[anchor = north] {$\bs{\Phi}^*_{i2},$};
\draw (0.4,-1.3) node {$\ldots$};
\draw (0.4,-0.4) -- (1,-1) node[anchor = north west] {$\bs{\Phi}^*_{i \ell_i}$ };
\draw (5.4,-1.1) node {\text{ Linearly independent in } $\Yqss$,};
\draw (5.6,-1.6) node {\text{ and } $\ell_i$ \text{ is the geometric multiplicity of } $\lambda_i$.}; 
\end{tikzpicture}
\end{equation*}
\nin The critical consequence of Theorem \ref{1.2} (actually, its adjoint version whose proof is essentially the same) is

\begin{thm}\label{I-Thm-3.1}
	\begin{enumerate}[(i)]
		\item With reference to \eqref{3.17}, we have for each $i$, 
		\begin{multline}\label{3.18}
		\text{the vectors } \bs{\Phi}_{i1}^*, \dots, \bs{\Phi}_{i\ell_i}^* \text{ remain linearly independent in } \\ \bs{L}^{q'}_{\sigma}(\omega) \times \bs{L}^{q'}_{\sigma}(\omega) \equiv \Yqss = \Yps, \ \frac{1}{q} + \frac{1}{q'} = 1.
		\end{multline}
		\item Consequently. It is possible to select vectors $\bs{u}_1,\dots,\bs{u}_K \in \Ls(\omega) \times \Ls(\omega),\ \bs{u}_i = [u_i^1, u^2_i], \ q > 1,\ K = \sup \ \{\ell_i:i = 1,\dots,M \}$, such that 
		\begin{subequations}\label{3.19}
			\begin{align}
			\text{ rank}
			\begin{bmatrix}
			(\bs{u}_1,\bs{\Phi}_{i1}^*)_{\omega} & \dots & (\bs{u}_K,\bs{\Phi}_{i1}^*)_{\omega} \\[1mm]
			(\bs{u}_1,\bs{\Phi}_{i2}^*)_{\omega} & \dots & (\bs{u}_K,\bs{\Phi}_{i2}^*)_{\omega} \\[1mm]
			\vdots &  & \vdots\\[1mm]
			(\bs{u}_1,\bs{\Phi}_{i \ell_i}^*)_{\omega}& \dots & (\bs{u}_K,\bs{\Phi}_{i \ell_i}^*)_{\omega} \\
			\end{bmatrix}
			&= \ell_i; \ \ell_i \times K \ \text{for each } i = 1,\dots,M, \label{3.19a}\\
			\left(\bs{u}_j, \bs{\Phi}_{i1}^* \right)_{\omega} = \left( \bbm u^1_j \\[1mm] u^2_j \ebm, \bbm \varphi_{ij}^* \\[1mm] \psi_{ij}^* \ebm \right)_{\bs{L}^q_{\sigma}(\omega) \times \bs{L}^q_{\sigma}(\omega)}, \ \ell_i &= \text{ geometric multiplicity of } \lambda_i. \label{3.19b}
			\end{align}
		\end{subequations}
	\end{enumerate}
\end{thm}

\begin{proof}
	\begin{enumerate}[(i)]
		\item By contradiction, let us assume that the vectors $\{\bs{\Phi}_{i1}^*, \dots, \bs{\Phi}_{i \ell_i}^*\}$ are instead linearly dependent on $\bs{L}^{q'}_{\sigma}(\omega) \times \bs{L}^{q'}_{\sigma}(\omega)$ , so that 		
		\begin{equation}\label{3.20}
		\bs{\Phi}_{i \ell_i}^* = \sum_{j = 1}^{\ell_i - 1} \alpha_j \bs{\Phi}_{i \ell_j}^* \text{ in } \Lqpso \times \Lqpso.
		\end{equation}
		Define the following function (depending on $i$) in $\Lo{q'} \times \Lo{q'}$		
		\begin{equation}\label{3.21}
		\bs{\Phi}^* = \left[ \sum_{j=1}^{\ell_i - 1} \alpha_j \bs{\Phi}_{i \ell_j}^* - \bs{\Phi}_{i \ell_i}^* \right] \in \Lo{q'} \times \Lo{q'}, \quad i = 1,\dots,M.
		\end{equation}
		\nin so that $\bs{\Phi}^* \equiv 0$ in $\omega$ by \eqref{3.20}. As each $\bs{\Phi}_{ij}^*$ is an eigenvalue of $\wti \BA_q^*$ (or $(\wti \BA_{q,N}^u)^*$) corresponding to the eigenvalue $\bar{\lambda}_i$, see \eqref{3.17}, so is the linear combination $\bs{\Phi}^*$. This property, along with $\bs{\Phi}^* \equiv 0$ in $\omega$ yields that $\bs{\Phi}^*$ satisfies the following overdetermined eigenvalue problem for the operator $\wti \BA^*_q$ (or $(\wti \BA_{q,N}^u)^*$),		
		\begin{equation}\label{3.22}
		\wti \BA_q^* \bs{\Phi}^* = \bar{\lambda} \bs{\Phi}^*, \text{ div } \bs{\Phi}^* = 0 \text{ in } \Omega; \quad \bs{\Phi}^* = 0 \text{ in } \omega, \text{ by \eqref{3.20}}.
		\end{equation}		
		\nin But the linear combination $\bs{\Phi}^*$ in \eqref{3.25} of the eigenfunctions $\bs{\Phi}_{ij}^* \in \calD(\wti \BA_q^*)$ satisfies itself the Dirichlet B.C $\ds \left. \bs{\Phi}^* \right \rvert_{\partial \Omega} = 0$. Thus the explicit PDE version of problem \eqref{3.22} with $\bs{\Phi}^* = \{ \varphi^*, \xi^* \}$ is
		\begin{subequations}\label{3.23}
			\begin{empheq}[left=\empheqlbrace]{align}
			- \nu \Delta \varphi^* + \calL^*_1 \varphi^* - \calL^*_2 \xi^* + \nabla p &= \bar \lambda \varphi^* & \text{ in } \Omega, \label{3.27a}\\
			-\eta \Delta \xi^* + \calM^*_1 \xi^* - \calM^*_2 \varphi^* &= \bar \lambda \xi^* & \text{ in } \Omega, \label{3.27b}\\
			\div \ \varphi^* \equiv 0, \quad \div \ \xi^* &\equiv 0 & \text{ in } \Omega,\\
			\varphi^* \equiv 0, \quad \xi^* \cdot n \equiv 0, \quad \curl \ \xi^* \times n &\equiv 0 & \text{ on } \Gamma,\\
			\varphi^* \equiv 0, \quad \xi^* &\equiv 0 &\text{ in } \omega, \label{3.23e}
			\end{empheq}
		\end{subequations}
		\begin{equation}\label{3.24}
		\bs{\Phi}^* \in \calD(\wti \BA_{q}^*); \ \calL_i^*\varphi^* = (y_e \cdot \nabla) \varphi^* + (\varphi^* \cdot \nabla)^* y_e,
		\end{equation}
		\nin with overdetermined conditions \eqref{3.23e}, where $(f.\nabla)^*y_e$ is a $d$-vector whose $i$\textsuperscript{th} component is $\displaystyle \sum_{j=1}^{d} (D_i y_{e_j})f_j$. The adjoint version of Theorem \ref{Thm-UCP-1} (essentially with the same proof) implies
		\begin{equation}\label{3.25}
		\varphi^* \equiv 0, \quad \xi^* \equiv 0, \quad p^* \equiv const \text{ in } \Omega; \text{ or } \bs{\Phi}^* = 0 \text{ in } \Lso \times \Lso.
		\end{equation}
		\begin{equation}\label{3.26}
		\text{ that is by } \eqref{3.21} \ \bs{\Phi}_{i \ell_i}^* = \alpha_1 \bs{\Phi}_{i1}^* + \alpha_2 \bs{\Phi}_{i2}^* + \dots + \alpha_{\ell_i - 1}\bs{\Phi}_{i \ell_i - 1}^* \text{ in } \Lo{q'} \times \Lo{q'},
		\end{equation}
		\noindent i.e. the set $\{ \bs{\Phi}_{i1}^*, \dots, \bs{\Phi}_{i \ell_i}^* \}$ in linearly dependent on $\Lo{q'} \times \Lo{q'}$. But this is false, by the very selection of such eigenvectors, see \eqref{3.26} and statement preceding it. Thus, the condition \eqref{3.20} cannot hold.
	\end{enumerate}
	\begin{rmk}
		Condition \eqref{3.19} is the Kalman algebraic condition for asserting controllability of the finite dimensional unstable component $\bos{\eta}_N$ of the $\ds \bbm w \\ \BW \ebm = \bos{\eta}$-dynamics, $\ds \bos{\eta} = \bos{\eta}_N + \bos{\zeta}_N$. It is then equivalent to the arbitrary spectrum location property \cite{BM:1992} and hence it implies that such originally unstable finite dimensional dynamics can be stabilized with an arbitrary large decay rate by a (finite dimensional) state feedback control.
	\end{rmk}
\end{proof}

\section*{Declarations}

\subsection*{Funding}
The research of I. L. was partially supported by  NSF Grant DMS-2205508 and by the NCN, grant Opus, Agreement UMO-2023/49/B/STI/04261. The research of R. T. was partially supported by NSF Grant DMS-2205508. The research of B. P. was partially supported by YSF offered by University of Konstanz under the project number: FP 638/23.

\subsection*{Competing Interests}
The authors have no competing interests to declare that are relevant to the content of this article.

\section*{Data availability statement}
This study did not involve the use of any datasets.

\end{document}